\documentclass{article}
\usepackage[style=numeric,maxnames=99,sortcites=true]{biblatex}
\usepackage{hyperref}
\usepackage{amsmath}
\usepackage{bbm}
\usepackage{tikz}
\usepackage{tikz-cd}
\usepackage{amsthm}
\usepackage{caption}
\usepackage{ amssymb }
\newtheorem{Theorem}{Theorem}[section]

\mathchardef\standardl=\mathcode`l
\mathcode`l=\ell
\newcommand{\deactivatel}{\mathcode`l=\standardl}
\makeatletter
\edef\operator@font{\operator@font\noexpand\deactivatel}
\makeatother

\title{Computing Non-Repetitive Sequences with a Computable Lefthanded Local Lemma}
\author{Daniel Mourad}
\theoremstyle{plain}
\newtheorem{theorem}[Theorem]{Theorem}
\newtheorem{Proposition}[Theorem]{Proposition}
\newtheorem{Lemma}[Theorem]{Lemma}

\newtheorem{Question}{Question}
\theoremstyle{definition}
\newtheorem{Definition}[Theorem]{Definition}

\theoremstyle{plain}

\newtheorem{Claim}[Theorem]{Claim}

\DeclareMathOperator{\ran}{ran}

\DeclareMathOperator{\vbl}{VBL}
\DeclareMathOperator{\rsp}{RSP}
\DeclareMathOperator{\stc}{STC}
\DeclareMathOperator{\Prob}{Pr}
\DeclareMathOperator{\ch}{ch}

\DeclareMathOperator{\cch}{cch}
\let\log\relax
\DeclareMathOperator{\log}{log}

\newcommand{\hlog}{\widehat{\log}}
\newcommand{\mc}[1]{\mathcal{\uppercase{#1}}}
\newcommand{\wh}[1]{\widehat{\uppercase{#1}}}
\newcommand{\bb}[1]{\mathbb{\uppercase{#1}}}
\newcommand{\mb}[1]{\mathbb{\uppercase{#1}}}
\newcommand{\nbr}{\Gamma}
\newcommand{\N}{\mb{N}}
\newcommand{\harpoon}{\upharpoonright}
\newcommand{\converges}{{\downarrow}}

\newcommand{\norm}[1]{\left\lVert#1\right\rVert}
\begin{document}
\maketitle
\begin{abstract}
	The lefthanded Lov\'asz local lemma (LLLL) is a generalization of the Lov\'asz local lemma (LLL), a powerful technique from the probabilistic method. We prove a computable version of the LLLL and use it to effectivize a collection of results on the existence of certain types of non-repetitive sequences via the LLL and LLLL. This represents the first constructive proof of these results. 
\end{abstract}

\section{Introduction}
The Lov\'asz local lemma (LLL) \cite{ErdosLovasz1975} is a theorem from combinatorics that generalizes the fact that if $\mc{A}$ is a finite set of mutually independent events such that each $A \in \mc{A}$ has positive probability, then the probability of their intersection is non-zero. 
In exchange for stricter requirements on the probabilities of the events, the LLL allows us to draw the same conclusion when $\mc{A}$ is \textit{not} a mutually independent set of events. For a directed graph $G$ and vertex $v \in G$, let $\nbr(v) = \{w \in G: v \rightarrow w\}$ be the set of out-neighbors of $v$. Let $\nbr^+(v) = \nbr(v) \cup \{v\}$ be the inclusive out neighbors of $v$.  Most applications of the LLL are framed in terms of avoiding ``bad'' events, hence we change the conclusion to say that the intersection of the \textit{complements} of the events is non-zero. 
\begin{Theorem}[Lov\'asz Local Lemma, General Form]
	\label{thm:LLL}
	Let $\mathcal{A}$ be a finite set of events in some probability space. Suppose there exists a directed graph $G$ on $\mathcal{A}$ and a real-valued function $z: \mathcal{A} \to (0,1)$ such that, for each $A \in \mathcal{A}$,
	\begin{enumerate}
		\item \label{cond:Ind} $A$ is mutually independent from $\mathcal{A} \setminus \nbr^+(A)$ and
		\item \label{cond:LLL} \[\Pr(A) \leq z(A) \prod_{B \in \Gamma(A)}(1 - z(B)).\]
	\end{enumerate}
	Then,
	\[
	\Pr\left(\bigcap_{A \in \mathcal{A}} \bar{A}\right) \geq \prod_{A \in \mathcal{A}}(1 - z(A)) > 0.
	\]
\end{Theorem}

The original formulation of the LLL by \citeauthor{ErdosLovasz1975} \cite{ErdosLovasz1975} applies only when $\mathcal{A}$ is finite. However, when used in combinatorics, the non-emptiness of  $\bigcap_{A \in \mathcal{A}} \bar{A}$ is the main point of interest. In this context the LLL can be made to generalize to infinite sets $\mathcal{A}$ of events if $\Omega$ is compact and each $A \in \mathcal{A}$ is open (so that $\bar{A}$ is closed).

However, in many cases of interest, this use of compactness is non-constructive. In fact, the original proof of the LLL is non-constructive for finite objects as well. In cases where $\Omega$ is finite, Erd\H{o}s and Lov\'asz's proof of the LLL gives no hint as to how to find an element of $\bigcap_{A \in \mathcal{A}} \bar{A}$ in a way faster than a brute force search. This problem has been of great interest in combinatorics and computer science \cite{Beck1991,Alon1991,Molloy1998,Czumaj2000,Srinivasan2008}. In a seminal result, Moser and Tardos \cite{Moser2010} give a simple and efficient constructive version of the LLL. Their constructive version applies to events in a \textit{variable context}, a collection $\mc{x}$ of mutually independent random variables such that each event $A \in \mc{A}$ is determined by the values of a finite subset $\vbl(A) \subset \mc{X}$ of variables. They use a probabilistic algorithm known as the \textit{resample algorithm}: start with a random sample of the variables in $\mc{X}$. Then, pick an $A \in \mc{A}$ which is true and then resample the variables in $\vbl(A)$. Repeat the latter step until each $A \in \mc{A}$ is false. Moser and Tardos' analysis of this algorithm proceeds by constructing trees, known as \textit{witness trees} or \textit{Moser trees} which record the resamplings leading up to a specific resampling. 

The algorithm of Moser and Tardos has been used as a foundation for many generalizations and improvements \cite{HeLiSun2023,Harvey2020,Pegden2014,Kolipaka2012,Achlioptas2012,Kolipaka2011,HarrisSrinivasan2019}, including a computable version by Rumyantsev and Shen \cite{Rumyantsev2013}. This effective version has been applied to multiple problems in computability theory and reverse math \cite{Reitzes2022,Liu,Csima2019}. 

However, this computable version of the LLL does not readily apply to all applications of the infinite LLL. One such case is a previously non-constructive theorem of Beck on the existence of nonrepetitive sequences. To describe intervals of indexed variables, we adopt the notation of \cite{GrytczukPryzybyloZhu2011}, in which $x[i,j) = x_i,\dots,x_{j-1}$, and similarly for other combinations of inclusive and non-inclusive endpoints.
	\newtheorem*{BecksTheorem}{Theorem~\ref{thm:Beck}}
\begin{BecksTheorem}[\citeauthor{Beck1981}\cite{Beck1981}]
	For every $\epsilon > 0$ there is an $N_\epsilon$ such that Player 1 has a strategy in the binary sequence game ensuring that any two identical blocks $x[i,i + n) = x_i x_{i+1} x_{i+2}\dots x_{i+n-1}$ and $x[j,j + n) = x_j x_{j+1} x_{j+2}\dots x_{j+n-1}$ of length $n > N_\epsilon$ have distance at least $f(n) = (2-\epsilon)^{n/2}$. That is, if $x_{i+s} = x_{j+s}$ for all $0\leq s <n$ and $n > N_\epsilon$, then $|i -j| < (2-\epsilon)^{n/2}$.  
\end{BecksTheorem}

The reason it is unclear whether Rumyantsev and Shen's computable LLL can be used to effectivise Theorem~\ref{thm:Beck} is because the computable LLL requires that each event has finitely many neighbors. In the proof of Theorem \ref{thm:Beck}, we define bad event $A_{k,\ell,n}$ for $(\ell - k) < (2-\epsilon)^n$ to be the event that the blocks $[k,k+n)$ and $[\ell,\ell+n)$ are identical. Then, $\vbl(A_{k,\ell,n}) = x[k,k+n) \cup x[\ell,\ell+n)$. However, we then have that $A_{k,\ell,n}$ is a neighbor of $A_{k,\ell + m, n +m}$ for all $m \geq 0$. 

In this paper, we modify the Moser--Tardos algorithm by resampling only a subset of the variables in $\vbl(A)$ to obtain computable versions of theorems such as Theorem~\ref{thm:Beck}. At the same time, we extend the reach of Moser and Tardos' analysis of the algorithm by introducing an order $\prec$ in which events are prioritized. In our analysis, we leverage this priority to extract more information about the order of resamplings from Moser trees. The introduction of a priority for resampling is inspired by a generalization of the LLL called the \textit{lefthanded Lov\'asz local lemma} (LLLL). This modified algorithm provides both constructive and computable versions of the LLL and LLLL. 

We use the modified algorithm to effectivize previously non-constructive applications of the LLL, one already mentioned due to Beck \cite{Beck1981} (Theorem \ref{thm:Beck}) and another due to \citeauthor{Alon1992} \cite{Alon1992} (Theorem \ref{thm:Alon}) on the existence of certain types of highly non-repetitive sequences. We also effectivise Pegden's \cite{Pegden2011} application of LLLL to game versions of Theorems \ref{thm:Beck} and \ref{thm:Alon}. Additionally, we effectivize applications of the LLLL due to \citeauthor{GrytczukPryzybyloZhu2011} \cite{GrytczukPryzybyloZhu2011} (Theorem~\ref{ThueChoice} ; and \citeauthor{PeresSchlag2010} \cite{PeresSchlag2010} (Theorem~\ref{PeresSchlagThm} ).

The LLLL allows for stronger interdependence between the events in exchange for the existence of a partial order on the events with certain properties. 

\begin{Theorem}[LLLL \cite{Pegden2011}]
	\label{LLLL}
	Let $\mc{A}$ be a finite set of events. Let $G$ be a directed graph on $\mc{A}$ endowed with a partial quasi-order $\prec$ such that 
	\begin{equation}
		\label{LOCondition}
		\begin{split}
			\text{if } B \in \nbr^+(A) &\text{ and } C \not \in \nbr(A), \text{ then}\\
			C \succ B &\text{ implies } C \succ A.
		\end{split}
	\end{equation}
	Let $P^*: \mc{A} \to (0,1)$ be a function such that for each $A \in \mc{A}$ and each set of events $\mc{B} \subset \mc{A} \setminus \nbr^+(A)$ satisfying $B \not\succ A$ for all $B \in \mc{B}$, we have 
	\[
		\Pr\left(A\bigg| \bigcap_{B \in \mc{B}} \bar{B}\right) \leq P^*(A). 
	\]
	
	Further suppose that there is a function $z: \mc{A} \to (0,1)$ such that, for each $A \in \mc{A}$,
	\[
		P^*(A) \leq z(A) \prod_{B \in \nbr(A)}\left(1 - z(B)\right).
	\]
	Then, 
	\[
		\Pr\left(\bigcap_{A \in \mathcal{A}} \bar{A}\right) > 0.
	\]
\end{Theorem}

Theorem \ref{CLLLL} makes a similar improvement to the results of Moser and Tardos; and Rumyantsev and Shen, allowing stronger interdependence in exchange for the existence of a linear order (a computable linear order in the computable version) with certain properties. In Theorem~\ref{CLLLL}, we interpret the linear order as determining the priority of the events, and Condition~\ref{LOCondition} as a requirement that, before ``moving on'' to non-neighbors of $A$ of higher priority than neighbors of $A$, we must first process $A$. 

\subsection{A Similar Resample Algorithm}

The modified resample algorithm is a special case of the \textit{partial resample algorithm} (although the analysis and results are incomparable) of \citeauthor{HarrisSrinivasan2019} \cite{HarrisSrinivasan2019},  in which a subset of $\vbl(A)$ is chosen at random according to a chosen probability distribution $Q_A$ and resampled. To obtain our modified resample algorithm, one may set the support of $Q_A$ to be $\{\rsp(A)\}$. The partial resample algorithm also split events into atomic events $B = \{x_{i_1} = \ell_{j_1} \wedge x_{i_2} = \ell_{j_2} \dots x_{i_m} = \ell_{j_m}\}$. In the case of Beck's Theorem, the natural way to run the partial resample algorithm would be to split $A_{k,\ell,n}$ into the atomic events $A_{k,\ell,n}^\sigma = \{x[k,k + n) = x[\ell,\ell+n) = \sigma\}$ for each $\sigma \in 2^n$. Furthermore, their analogue to Condition~\ref{cond:LLL} uses $\Pr(\rsp(A_{k,\ell,n}^\sigma) = \sigma)$ instead of $\Pr(A_{k,\ell,n}^\sigma)$. 

This results in events having $2^n$ many more neighbors without decreasing their probabilities, making it unclear how choose appropriate $z_i$ to fulfill the condition analogous to Condition~\ref{cond:LLL}. Our analysis avoids this problem by remaining at the level of events. For the sake of clarifying the relationship between these results, we note that there is a possibility that the methods of \cite{HarrisSrinivasan2019} and the present paper may be combined. In the example of Beck's Theorem, this may look like having the probabilistic bound be based on $\Pr(A_{k,\ell,n}^\sigma)$ instead of $\Pr(\rsp(A_{k,\ell,n}^\sigma) = \sigma)$. We leave this to future work.

\section{Constructive and Computable Versions of the Lefthanded Local Lemma}

Recall that an event $A$ being in the sigma-algebra $\sigma(\mc{X})$ generated  by a set $\mc{X}$ of random variables is equivalent to the truth or falsity of $A$ being completely determined by the outcomes of the variables in $\mc{X}$. 

Let $\mc{X}$ be a mutually independent set of countably many random variables with finite ranges. 
Let $\mc{A}$ be a finite set of events in the sigma-algebra $\sigma(\mc{X})$ generated by $\mc{X}$ such that each $A \in \mc{A}$ has a finite set $\vbl(A) \subset \mc{X}$ such that $A \in \sigma(\vbl(A))$. 
For each $A \in \mc{A}$, let $\rsp(A) \subset\vbl(A)$ be a set of variables that we will resample when $A$ is bad. 
Let $\stc(A) =  \vbl(A)\setminus \rsp(A)$ be the \textit{static} variables of $A$. 
For the computable version of the lemma, we will require for each $x \in \mc{X}$ that $\{A:x \in \rsp(A)\}$ is finite.

For each $A \in \mc{A}$, let 
\[
\nbr(A) = \{B \in \mc{A}: \rsp(A)  \cap \rsp(B) \neq \emptyset \text{ and } A \neq B\}
\]
be the neighborhood of $A$. Let $\nbr^+(A) = \nbr(A) \cup {A}$. 

Let $\prec$ be a linear order on $\mc{A}$ with type $\omega$ such that, for all $A, B, C \in \mc{A}$, 
\begin{equation}
	\label{LefthandedRule}
	\begin{split}
	\text{if } A \in \nbr^+(B) &\text{ and } A \not\in \nbr(C), \text{ then }\\ 
	C \succ B &\text{ implies } C \succ A. 
	\end{split}
\end{equation}

We can think of Condition \ref{LefthandedRule} as a requirement that, prior to addressing any event $C \succ B$, we first make sure that all shared neighbors of $B$ and $C$ have been taken care of. 

If $A \prec B$ and $A \not\in \nbr(B)$, then we write $A \ll B$. We also require that if $A \gg B$ and $A \not \in \nbr(B)$, then $\rsp(A) \cap \vbl(B) = \emptyset$. We can think of this condition as requiring that we pick the $\rsp$ sets and order $\prec$ in such a way so that the variables in $\stc(A)$ have lower priority than the variables in $\rsp(A)$.

	All currently known applications of the LLLL which are not also an application of a less general form the of LLL admit a variable context in which $\vbl(A)$ is an interval in the variable set $\mc{X} = \{x_1,x_2,\dots\}$ and use the same partial order. Recall the notation $x[i,j) = x_i,\dots,x_{j-1}$, and similarly for other combinations of inclusive and non-inclusive endpoints. Such applications typically use the same linear order $\prec$. 

	\begin{Lemma}
		\label{IntervalsWork}
		Suppose that each $A \in \mc{A}$ is such that $\vbl(A) = x[i_A,r_A]$ and $\rsp(A) = x[j_A,r_A]$ for some $i_A \leq j_A \leq r_A$. Let $\prec^*$ be the quasi order defined by $A \prec^* B$ if and only if $r_A \leq r_B$. Let $\prec$ be any linearization of $\prec^*$. Then, $\prec$ is a linear order of type $\omega$ satisfying Condition~\ref{LefthandedRule} and we also have that $A \gg B$ implies that $\rsp(A) \cap \vbl(B) = \emptyset$.
	\end{Lemma}

We derive some important properties of $\ll$, $\prec$, and $\nbr$. 

\begin{Proposition} The following are true. 
		\begin{enumerate}
		\label{LefthandedProps}
		
		\item For each $A,B \in \mc{A}$, exactly one of the following hold.
		\begin{enumerate}
			\item $A \gg B$
			\item $A \in \nbr^+(B)$
			\item $B \gg A$
		\end{enumerate}
		\item If $B \in \nbr^+(A)$ and $C \ll A$, then $C \prec B$. 
		\item The relation $\ll$ is transitive
	\end{enumerate}
\end{Proposition}

\begin{proof} Suppose $\ll$, $\prec$, and $\nbr$ are as above. 
	\begin{enumerate}
		\item If $A \not \in \nbr^+(B)$, then $A \neq B$, so either $A \prec B$ or $B \prec A$. 
		\item We have that $A \in \nbr(B)$ and $A \not \in \nbr(C)$. Then, by Condition \ref{LefthandedRule}, we have that $ C\succ B$ implies that $C \succ A$. Since $C \prec A$, we conclude that $C \prec B$. 		
		\item Suppose that $A \gg C$ and $C \gg B$. Suppose that $A \in \nbr(B)$. Then, by part 2 of this Proposition, we obtain that $C \prec B$, which contradicts that $C \gg B$. 
	\end{enumerate}
\end{proof}

Now we define the modified resample algorithm. The modified resample algorithm operates in stages. At stage $s$, the modified resample algorithm produces a valuation $R_s$ of the variables in $\mc{X}$. We also keep track of the events resampled, forming a log $E_1,E_2,\dots$ of events in $\mc{A}$.

Stage $0$: For each $x \in \mc{X}$, let $R_0(x)$ be a random sample of $x$ in a way such that the set of random variables $\{R_0(x): x \in \mc{X}\}$ is mutually independent. 

Stage $n+1$: Let $A \in \mc{A}$ be the $\prec$-least event in $\mc{A}$ which is false under the valuation $R_n$. If there is no such event $A$, then halt. Otherwise, let $E_{n+1} = A$. For each $x \in \rsp(A)$, let $R_{n+1}(x)$ be a random sample of $x$ which is mutually independent of all the samplings taken so far. For each $x \in \mc{X} \setminus \rsp(A)$, let $R_{n+1}(x) = R_{n}(x)$. 

Theorem \ref{ALLLL} gives conditions which imply that the modified resample algorithm reliably makes progress finding a valuation of $\mc{X}$ which makes each $A \in \mc{A}$ false. 

For a valuation $\mu$  of some subset $S \subset \mc{X}$ of variables, let $E_\mu$ be the event that, for all  $x \in S$, we have $\mu(x) = x$. Here, when we write $\mu(x) = x$, we abuse notation by having $x$ play the role of a variable on the left hand side of the equation and having $x$ play the role of a random variable on the right hand side of the equation. 

For $A \in \mc{A}$, let $\mc{E}_A$ be the set of valuations of $\stc(A)$. 

\begin{Theorem}
	\label{ALLLL}
	Suppose there is a $P^*: \mc{A} \to (0,1)$ and a $z: \mc{A} \to (0,1)$ such that, for each $A \in \mc{A}$ and $\mu \in \mc{e}_A$,
	\begin{equation} 
		\label{PStarCondition}
		P^*(A) \geq \sup_{\mu \in \mc{E}_A} ((\Prob(A | E_\mu))
	\end{equation}

	and
	\[
		P^*(A) \leq z(A)\prod_{B \in \nbr(A)} \left(1 - z(B)\right).
	\]
	
	Then, for any finite set of events $B \subset \mc{A}$, the stopping time $t_\mc{B} =$ ``the first stage of the modified resample algorithm at which all $B \in \mc{b}$ are false'' has finite expectation. 
\end{Theorem}

For the computable version of Theorem \ref{ALLLL}, we introduce some computability conditions. We say that a countable set of finite sets $X_1,X_2,\dots$ is computably presented if the G\"odel code of $X_i$ is uniformly computable with respect to $i$. 

\begin{Theorem}
	\label{CLLLL}
	Suppose that a set of variables $\mc{X} = \{x_1,x_2,\dots\}$ and events $\mc{A} = \{A_1,A_2,\dots\}$ satisfy the conditions of Theorem \ref{ALLLL}. Assume the following conditions.
	\begin{itemize}
	\item The sets $\rsp(A_i)$, $\vbl(A_i)$, and the finite set of assignments of the variables in $\vbl(A_i)$ that make $A_i$ true are all computably presented.
	\item Each $x_i$ has a rational-valued probability distribution that is uniformly computable with respect to $i$. Also, the set $\{j: x_i \in \rsp(A_j)\}$ of indices of events which can resample $x_i$ is both finite and computably presented.
	\item The order $\prec$ is computable.
	\end{itemize}
	Additionally, require that there is a rational constant $\alpha \in (0,1)$ such that for all $A \in \mc{a}$, 
	
	\begin{equation}
		\label{LeftLocalLemmaCondition}
		P^*(A) \leq \alpha z(A)\prod_{B \in \nbr(A)} \left(1 - z(B)\right).
	\end{equation}
	
	Then, there is a computable valuation of $\mc{x}$ which makes each $A \in \mc{A}$ true. 
\end{Theorem}

Note that, while the LLLL (Theorem~\ref{LLLL}) allows its neighborhood graph to be directed, the graph induced by $\nbr$ in Theorem~\ref{CLLLL} is always undirected. As a result, applications of Theorem~\ref{CLLLL} typically involve neighborhood sets twice as large as those in applications of Theorem~\ref{LLLL}. In extremal applications, this is not significant. However, when optimization is desired, we end up with a gap between constructive and non-constructive versions of the same theorem (see Theorem~\ref{ComputableThueChoice}). Additionally, note that Condition~\ref{PStarCondition} is somewhat stronger than the corresponding condition in the LLLL, although this difference has not yet been relevant in any applications.  

The constant $\alpha$ in Condition~\ref{LeftLocalLemmaCondition} is technically important but also is not an obstacle in any known applications. 

\section{Analysis of the Modified Resample Algorithm} 

\subsection{The Log of the Resample Algorithm and Moser Trees} 

We begin the proof of Theorem \ref{ALLLL} with analysis of the sequence of events resampled by the modified resample algorithm, henceforth referred to as the resample algorithm. 

\begin{Definition}
	A sequence $E_1, E_2, \dots $ of events from $\mc{a}$ is a legal log if, for each $1\leq i < n$, we have that $E_i \not\gg E_{i+1}$. 
\end{Definition}

Legal logs characterize the logs of the resample algorithm. 

\begin{Claim}
	Every sequence of events resampled by the resample algorithm is a legal log. 
\end{Claim}

\begin{proof}
	Let $E_1, \dots, E_n$ be an initial segment of the log of a run of the resample algorithm. Suppose that $E_i \gg E_{i+1}$. Then $\rsp(E_{i}) \cap \vbl(E_{i+1}) = \emptyset$. Also, since $E_i \prec E_{i+1}$, it must be that $E_{i+1}$ is false at stage $i$. But then, $E_{i+1}$ went from false to true after resampling $E_i$, which contradicts that $\rsp(E_{i}) \cap \vbl(E_{i+1}) = \emptyset$. 
\end{proof}

We will use the following fact. 

\begin{Proposition}
	\label{LoseProgressAlongNeighbors}
	Let $E_1, E_2, \dots, E_n$ be a legal log. Suppose that $E_i \succ E_j$ for some $i < j$. Then, there is $k$ with $i \leq k < j$ such that $E_k \in \nbr^+(E_j)$. 	
\end{Proposition}
\begin{proof}
	If $E_i \in \nbr^+(E_j)$ then we are done. Otherwise, $E_i \gg E_j$. Since $E_1,\dots, E_n$ is a legal log, there is some $k$ (in particular, $k = j-1$) such that $i < k < j$  and $E_k \not \gg E_j$. If $E_k \in \nbr^+ (E_j)$, we are done. Otherwise, $E_k \ll E_j$.  Then, we can fix the least $s > i$ such that $E_s \ll E_j$. Consider $E_{s-1}$. By minimality of $S$ and the fact that $E_i \gg E_j$, we have that $E_{s-1} \not \ll E_j$. Suppose that $E_{s-1} \gg E_j$. Since $E_j \gg E_s$, we have that $E_{s-1} \gg E_s$, contradicting that $E_1,\dots,E_n$ is a legal log. So, $E_{s-1} \in \nbr^+(E_j)$. 
\end{proof}

If we think of resampling $\prec$-greater events as ``progress'' in the resample algorithm then Proposition \ref{LoseProgressAlongNeighbors} tells us that, if we somehow lose progress, it must be due to resampling a neighbor. Keeping track of how we have lost progress is the core of the present argument. We record lost progress using labeled trees. 

\begin{Definition}
	Let $E_1,\dots E_n$ be a legal log. The Moser tree generated by $E_1,\dots, E_n$ is a finite tree $T$ with labels in $\mc{a}$ constructed as follows. 
	
	We start at stage $n$ and count backwards until stage $1$, defining $T_i$ at stage $i$, with $T = T_1 = \bigcup_{i \leq n}T_i$. For a node $x$, let $[x] \in \mc{a}$ denote the label of $x$ and let $d(x)$ denote depth of $x$, meaning the distance of $x$ to the root
	
	Stage $n$: $T_n$ has one vertex labeled $E_n$. 
	
	Stage $k-1$: Check if there is some $x \in T_k$ with $E_{k-1} \in \nbr([x])$. If not, then $T_{k-1} = T_k$. If there is, then pick an element $y \in \{x \in T_k: [x] \in \nbr(E_{k-1})\}$ of maximal depth. Let $T_{k-1}$ be $T_k$ together with a new child of $y$ labeled $E_{k-1}$. 
	
	When we refer to the collection of Moser trees, we mean the set of labeled trees which are generated by some finite legal log. 
\end{Definition}

	For $x \in T$, let $q(x)$ be the stage at which $x$ was added to $T$, that is, $\left(T_{q(x)} \setminus T_{q(x)-1}\right) = \{x\}$. We also abuse notation to let $r = |T|$ and also to let $r$ be the root node of $T$, e.g., $q(v_r) = q(r)$. 
	We state some basic properties of Moser trees.	
	\begin{Proposition}
		\label{BasicTreeProps}
		Let $E_1,\dots, E_n$ be a legal log. Then, the Moser tree $T$ generated by $E_1,\dots, E_n$ has the following properties. 
		\begin{enumerate}
			\item Suppose that $x,y \in T$, $[x] \in \nbr^+([y])$ and $x \neq y$. Then $d(x) \neq d(y)$. Furthermore, $d(x) > d(y)$ if and only if $q(x) < q(y)$. 
			\item $T$ is not the Moser tree generated by $E_1,\dots, E_m$ for any $m \neq n$. 
			\item Let $m < n$ and $R$ be the Moser tree generated by $E_r, E_{r+1},\dots, E_n$. Then, $R$ is a subset of  $T$ as a labeled tree. 
		\end{enumerate}
	\end{Proposition}
	\begin{proof}
		\begin{enumerate}
			\item Suppose that $x \neq y$, $[x] \in \nbr^+([y])$ and $q(x) < q(y)$. Then, $y \in T_{q(x)+1}$, so $x$ is added as a child of a node at least as deep as $y$. Hence $d(x) > d(y)$. By a symmetric argument, we obtain that $q(y) < q(x)$ implies $d(x)>d(y)$. Since $x \neq y$ implies that $q(x) \neq q(y)$, we obtain that $d(x) \neq d(y)$. 
			\item Recall that $r \in T$ is the root vertex of $T$. If $T$ were also the Moser tree associated with $E_1,E_2,\dots, E_m$, then $E_m = E_n = [r]$. We also have that $[r] \in \nbr^+([r])$, so there is a one to one correspondence between each instance of $[r]$ in $E_1,\dots,E_n$ and each node in $T$ labeled $[r]$. The same holds for $E_1,\dots,E_m$. Hence, $E_1,\dots,E_m$ has the same number of occurrences of $[r]$ as $E_1,\dots,E_n$. Since $E_n = [r]$, this implies that $n = m$. 
			\item $R \subset T$ because $R = T_m \subset T$. 
		\end{enumerate}
	\end{proof}

	Requiring that the resample algorithm resamples the $\prec$-least bad event fixes the order in which vertices are added to Moser trees; the most recently added vertex to a Moser tree $T$ is the $\prec$-least vertex which is of maximal depth among its neighbors. To be precise, let $M(T) = \{v \in T: d(v) = \max(\{d(w): [w] \in \nbr^+([v])\})\}$ be the set of vertices which are of maximal depth among their neighbors. By Proposition \ref{BasicTreeProps}, we have that for each $v,w \in M(T)$, if $v \neq w$ then $[v] \neq [w]$. Therefore, it is well defined to let $G(T) = v$ such that $[v] = \min_{\prec}({[w]: w \in M(T)})$.
	
	Let $v_1,v_2,\dots,v_{r}$ be defined recursively by $v_1 = G(T)$ and $v_{i+1} = G(T \setminus \{v_1,\dots,v_{i}\})$. 
	
	\begin{Claim}
		\label{OrderClaim}
		Suppose $T$ is the Moser tree generated by $E_1,\dots,E_n$. Then, for each $0 \leq i < r$, we have that
		\[
			q(v_{i+1}) = \min(\{q(v_j): j > i\}).
		\]
		Thus, we conclude that
		\[
			q(v_1) < q(v_2) < \dots < q(r).
		\]
	\end{Claim}
	\begin{proof}
		We proceed by strong induction on $i$. Suppose the claim is true for each $j < i$. Denote $K = T \setminus \{v_1,\dots,v_i\}$.		
		Let $v \in K$ be such that $q(v) = \min(\{q(w): w\in K\})$. We show that $v = G(K)$. Since $q(v)$ is minimal, $d(v)$ is maximal among the nodes whose labels are neighbors of $[v]$ by Proposition \ref{BasicTreeProps}, so $v \in M(K)$. It remains to show that $[v]$ is $\prec$-minimal in $\{A: (\exists w \in M(K))([w] = A)\}$. 
		
		Suppose $w \in M(K)$, $v \neq w$ and $[v] \not\prec [w]$. Since both $v,w \in M(K)$, we have that $[v] \neq [w]$. Hence, $[w] \succ [v].$ Since $q(v) < q(w)$, Proposition \ref{LoseProgressAlongNeighbors} says there is $k$ with $q(v) \leq k < q(v)$ such that $E_k \in \nbr^+{[v]}$. Then, a node $x$ is added to $K$ at stage $k$ of the Moser tree construction with $d(x) > d(w)$ and $[x] = E_k$. This contradicts that $w \in M(K)$.  
	\end{proof}
	
	The next lemma can be interpreted as saying that the Moser tree indeed records every instance of losing progress on the way to resampling $E_n$. 
	
	\begin{Proposition}
		\label{MoserTreeDoubleLessThan}
		Suppose legal log $E_1,\dots,E_n$ produces Moser tree $T$. Fix $i$ and $k$ such that $0 \leq i < r$ and $q(v_i) < k < q(v_{i+1})$ (let $q(v_0) = 0$). Then, $E_k \ll [v_j]$ for all $r \geq j > i$. 
	\end{Proposition}
	\begin{proof}
		Suppose that $E_k \not \ll [v_j]$ for some $j > i$. There are two cases. 
		
		Case 1: Suppose $E_k \in \nbr^+([v_j])$. Then, there is a new node added at stage $k$ of constructing the Moser tree. Since, $q(v_i) < k < q(v_{i+1})$ this contradicts Claim \ref{OrderClaim}. 
		
		Case 2: Suppose that $E_k \gg [v_i]$. Then, by Proposition \ref{LoseProgressAlongNeighbors}, there is $\ell$ with $k < \ell < q(v_{i+1})$ such that $E_\ell \in \nbr^+([v_j])$. Since $q(v_i) < \ell < q(v_{i+1})$ this also contradicts Claim \ref{OrderClaim}.
	\end{proof}
	
	\subsection{The T-check} 
	
		We now describe a random process which we will use in a coupling argument with the resample algorithm. 
	\begin{Definition}
		Let $T$ be a Moser tree. Like the resample algorithm, at each stage $s$, the $T$-check produces a valuation $\wh{R}_s$ of the variables in $\mc{x}$. The $T$-check also records a log $\wh{E}_1,\wh{E}_2,\dots$ of events resampled by the $T$-check. For each $x \in \mc{x}$, let $\{\wh{R}_0(x): x \in \mc{x}\}$ be a set of mutually independent random samplings of the variables in $\mc{x}$. The $T$-check proceeds in a series of $T \setminus \{v_1,\dots, v_i\}$-subchecks for each $0 \leq i \leq r$, beginning with the $T$-subcheck and stage number $0$. The subchecks are defined as follows. 
		
		For $i \neq r$, let $P = T \setminus \{v_1,\dots,v_i\}$ and let $\mc{a}_{\ll P} = \{A \in \mc{a}:(\forall v \in P) (A \ll [v])\}$. Suppose the $P$-subcheck starts at stage $s$. Then the $P$-subcheck runs the resample algorithm starting with the valuation $\wh{R}_s$ until each $A \in \mc{a}_{\ll P}$ is good. Whenever a stage of the resample algorithm is completed at stage $t$ of the $T$-check, the $P$-subcheck updates the log by setting $\wh{E}_t$ equal to the event which was resampled and then increments the stage number by $1$. If $\wh{R}_s$ makes each $A \in \mc{a}_{\ll P}$ good, then the $P$-subcheck resamples $[G(P)]$, sets $E_s = [G(P)]$ and then increments $s$ by one. Then, the $T$-check begins the $P \setminus \{G(P)\}$-subcheck.
		
		The $\emptyset$-subcheck halts. 
		
		For each $v_i \in T$, let $\hat{q}(v_i)$ be the stage in which the $T$-check begins the $T \setminus \{v_1,\dots,v_{i}\}$-subcheck. If this never happens, then $\hat{q}(v_i) = \infty.$ We say that the $T$-check passes if and only if, for each $v_i \in T$, we have that $\hat{q}(v_i) < \infty$ and the valuation $\hat{R}_{\hat{q}(v_i)}$ makes the event $[v_i]$ bad. In other words, the $T$-check passes when each $\hat{q(v_i)}$ is finite and $[v_i]$ is bad \textit{before} it is resampled at stage $\hat{q}(v_i)$. 
	\end{Definition}
	
	Each log of the $T$-check is a legal log.
	
	\begin{Claim}
		Let $\hlog = \wh{e}_1,\wh{e}_2, \dots$ be the log of a $T$-check. Then, $\wh{E}_1,\wh{e}_2,\dots$ is a legal log.
	\end{Claim}
	\begin{proof}
		
		Suppose $\wh{e}_{i+1} \ll \wh{e}_i$. Since all logs of the resample algorithm are legal logs, $i + 1$ must be one of the stages equal to $\hat{q}(v_j)$ for some $v_j \in T$ (as opposed to a non $\hat{q}$-stage in which the $T$-check mimics the resample algorithm). There are two cases:
		
		Case 1: $i = \hat{q}(v_{j-1})$. Then, $[v_{j-1}] \gg [v_j]$. Let $K = T \setminus \{v_1,\dots,v_{j-2}\}$. If $v_j \in M(K)$, then $\wh{e}_{i+1} = [v_j] \succeq [G(K)] = [v_{j-1}] = \wh{E}_{i}$, a contradiction. If $v_j \not \in M(K)$, then $v_{j-1}$ is a child of $v_j$, so $\wh{E}_{i} = v_{j-1} \in \nbr^+([v_{j}]) = \nbr^+(\wh{E}_{i+1})$, also a contradiction.
		
		Case 2: $i \neq \hat{q}(v_{j-1})$. Let $K = T \setminus \{v_1,\dots,v_{j-1}\}.$ Then, there must be some $A \in \mc{a}_{\ll K}$ which is bad at stage $i$. However, $A \ll [v_j] = \wh{e}_{i+1} \ll \wh{e}_i$, so the resample algorithm should have chosen to resample an event $\prec$-less than $A$ at stage $i$, a contradiction. 
	\end{proof}

	We will couple the $T$-check with the resample algorithm when run on the same random source. First, we specify our probability spaces. Let $\mc{x}' = \{x_i^j: i,j \in \mathbb{N}\}$ be a mutually independent set of random variables such that each $x_i^j$ is distributed identically to $x_i$. The set $\mc{x}'$ is essentially a collection of countably many copies of each variable in $\mc{x}$. Let $(\Omega, \Pr)$ be the probability space of all valuations of the variables in $\mc{x}$ and let $(\Omega', \Pr)$ be the probability space of all valuations of the variables in $\mc{x}'$. 
	
	Consider $S \in \Omega'$. We think of $S$ as a source for the randomness in either the $T$-check or the resample algorithm as follows.  
	For a given run of the resample algorithm, let $V_s(i)$ be the number of times $x_i$ has been resampled at the beginning of stage $s$. Similarly, let $\widehat{V}_s(i)$ be the number of times a given run of the $T$-check has been resampled at the beginning of stage $i$. Recall that, at stage $s$, the resample algorithm produces a valuation $R_s$ of the variables in $\mc{x}$ and the $T$-check produces a valuation $\wh{R}_s$. 
	
	We say that the resample algorithm is run with random source $S$ if, for each stage $s$ and variable $x_i \in \mc{x}$, we have that
	\[
		R_s(x_i) = S\left(x_i^{V_s(x_i)}\right).
	\]
	
	Similarly, we say that the resample algorithm is run with random source $S$ if, for each stage $s$ and variable $x_i \in \mc{x}$, we have that 
	\[
		\wh{R}_s(x_i) = S\left(x_i^{\wh{v}_s(x_i)}\right).
	\]
	
	Let $\log: \Omega' \to \mc{a}^{\leq \omega}$ be such that $\log(S) = E_1,E_2,\dots$ is the log of the resample algorithm when run with random source $S$. Let $\widehat{\log}: \Omega' \to \mc{a}^{\leq \omega}$ be such that $\widehat{\log}(S) = \hat{E}_1,\hat{E}_2,\dots$ is the log of the $T$-check when run with random source $S$. We also treat $q$, $\hat{q}$, $V_s$, $E_i$, $\wh{e}_i$ and $\widehat{V}_s$ as random variables. 
	
	\begin{Proposition}
		\label{SameTrees}
		Fix a Moser tree $T$. Suppose that the resample algorithm and the $T$-check are run using the same random source $S \in \Omega'$. Then, the $T$-check passes whenever an initial segment of the log of the resample algorithm produces $T$.
	\end{Proposition}
	\begin{proof}
		 Suppose $E_1,\dots,E_n$ is an initial segment of $\log(S)$ and that $T$ is the Moser tree generated by $E_1,\dots,E_n$. Let $\hlog(S) = \widehat{E}_1,\widehat{E}_2,\dots$ be the log of the $T$-check. We show that the $T$-check passes.
		
		 First we show that $\log(S) = \widehat{\log}(S)$. Since the $T$-check and the resample algorithm do the same thing except possibly at stages $\hat{q}(S)(v_i)$ for some $v_i \in T$, it suffices to show that for each $v_i \in T$ that $\hat{q}(S)(v_i) = q(S)(v_i)$. We proceed by strong induction on $i$.
		 
		 Let the expression $q(v_0)$ equal $0$ by definition. Suppose that $\hat{q}(S)(v_j) = q(S)(v_j)$ for each $j < i$. Then, $E_k = \wh{e}_k$ for each $k \leq q(v_{i-1})$. By Proposition \ref{MoserTreeDoubleLessThan}, for each $k$ with $q(S)(v_{i-1}) < k < q(S)(v_i)$, we have that $E_k \ll [v_i]$. One may show by induction on $k$ that this implies $E_k = \wh{E}_k$ for $k < q(S)(v_i)$. Since $\wh{e}_k \ll [v_i]$ for each $k < q(S)(v_i)$, we have that $\hat{q}(S)(v_i) \geq q(S)(v_i)$. 
		 
		 Let $K = T \setminus \{ v_1,\dots,v_i\}$. Since $E_{q(S)(v_i)} = [v_i] \not \in \mc{A}_{\ll K}$, it must be the case $R_{q(S)(v_i)}$ makes each $A \in \mc{a}_{\ll K}$ good, so $\hat{q}(S)(v_i) = q(S)(v_i)$. 
		
		The $T$-check passes because, for each $i$, we have that $E_{\hat{q}(S)(v_i)} = [v_i]$, which means that $R_{\hat{q}(S)(v_i)} = \wh{R}_{\hat{q}(S)(v_i)}$ makes $[v_i]$ bad. 
		
	\end{proof}
	
	Thus, we conclude that
	\[
	\begin{split}
	\Pr(T \text{ is generated by an initial segment of the} &\text{ resample algorithm}) \\&\leq \Pr(\text{The } T\text{-check passes}).
	\end{split}
	\]
	\subsection{Bounding the Probability the \texorpdfstring{$T$}{T}-check Passes} 
	
	In this section, we will show that 
	\begin{equation}
		\label{TCheckBound}
		\Pr(\text{The } T\text{-check passes}) \leq \prod_{v \in T} P^*([v]).
	\end{equation}
	
 	We develop some basic properties of the $T$-check. First we confirm that some relevant sets are measurable. 
	
	\begin{Proposition}
		Fix Moser tree $T$ of size $r$. The following sets are clopen in $(\Omega',\Pr)$, and therefore measurable.
		\begin{itemize}
			\item For any $\widehat{E}_1,\dots,\widehat{E}_n \in \mc{a}^n$, the set \[\{S \in \Omega': \hat{E}_1,\dots,\hat{E}_n \sqsubset \widehat{\log}(S)\}.\]
			\item For each $1 \leq i \leq r$ and $k \in \mathbb{N}$, the set \[\{S \in \Omega': \hat{q}(v_i) = k\}.\]
			\item For all $1 \leq i \leq r$, $k \in \mathbb{N}$, $x_n \in \vbl([v_i])$ and $\ell \in \ran([x_i])$, the set
			\[
			\{S \in \Omega': \hat{q}(S)(v_i) = k \text{ and } \widehat{V}_{\hat{q}(S)(v_i)}(S)(n) = \ell\}.
			\] 
		\end{itemize}
	\end{Proposition}
	\begin{proof}
		Each of these sets are decided true or false by a finite stage of the $T$-check, which uses finitely many variables in $\mc{X}'$. 
	\end{proof}
	
	Next, we show that whenever the $T$-check checks whether $[v_i]$ is true at stage $\hat{q}(v_i)$, the variables in $\rsp([v_i])$ will always have been resampled the same number of times. 

	\begin{Proposition}
		\label{ConstantVs}
		For each $v_i \in T$ and each $x_j \in \rsp([v_i])$, the number $\wh{v}_{\hat{q}(v_i)}(j)$ is constant as a random variable on the subspace $\{\hat{q}(r) < \infty\}$. 
	\end{Proposition}
	\begin{proof}
		Let $u(i,j) = \{s<i: x_j \in \rsp([v_s])\}$. Note that $u(k,j) \subset u(i,j)$ for all $k < i$. We show that $\wh{v}_{\hat{q}(v_i)}(j) = |u(i,j)|$ whenever $\hat{q}(v_i) < \infty$. We proceed by strong induction on $|u(i,j)|$. 
		
		Suppose $|u(i,j)| = \wh{v}_{\hat{q}(v_i)}(j)$ whenever $|u(i,j)| < n$. First we show that $\wh{v}_{\hat{q}(v_i)}(j) \geq n$. When $n = 0$, this is true by definition. Suppose that $|u(i,j)| = n > 0$. Then, $u(i,j) = \{k_1,\dots,k_n\}$ for some increasing sequence $k_1,\dots,k_n$. Since the sequence is increasing, $u(k_n,j) = \{k_1,\dots,k_{n-1}\}$ and $|u(k_n,j)| = n-1$. Hence, by the induction hypothesis, $|u(k_n,j)| = \wh{v}_{\hat{q}(v_{k_n})}(j) = n-1$. Since $x_j$ is resampled once more at stage $\hat{q}(v_{k_n})$, we have that $\wh{v}_{\hat{q}(v_i)}(j) \geq n$.
		
		Assume by way of contradiction that $\wh{v}_{\hat{q}(v_i)}(j) > n$. Then, $x_j$ is resampled at some stage $\ell$ such that $\hat{q}(v_{k_n}) < \ell < \hat{q}(v_i)$. Since $x_j \in \rsp(E_\ell)$, we have that $E_\ell \in \nbr^+([v_i])$. Fix $A \ll [v_i]$. By Proposition \ref{LefthandedProps}, we have that $E_\ell \succ A$, so $\wh{R}_\ell$ makes $A$ good. Let $K =  T \setminus \{v_1,\dots,v_{i-1}\}$. The fact that $A \ll [v_i]$ implies that $A \in \mc{a}_{\ll K}$, so each $A \in \mc{a}_{\ll K}$ is good at stage $\ell$. This contradicts that $\hat{q}(v_i) > \ell$. 
		
		Thus, $\wh{v}_{\hat{q}(v_i)}(j) = |u(i,j)|,$ which is constant in $\Omega'$.
	\end{proof}

	For $v \in T$, let  \[\mc{v}_v = \left\{x_i^{\wh{v}(\hat{q}(v))(i)}: x_i \in \rsp([v])\right\} \subset \mc{x}'\] be the set of variables from $\mc{x}'$ that the $T$-check uses as resample variables to determine whether or not $[v]$ is bad at stage $\hat{q}(v)$. Let $\mc{v} = \bigcup_{v\in T} \mc{v}_v$. The next claim states that the log of the resample algorithm is independent from $\mc{v}$
	
	\begin{Claim}
		\label{IndependentLogs}
		Let $S_1, S_2 \in \Omega'$ be such that $S_1(x_i^j) = S_2(x_i^j)$ for each $x_i^j \in \mc{x} \setminus \mc{v}$. Then, 
		\[
			\hlog(S_1) = \hlog(S_2).
		\]
		
	\end{Claim}
	\begin{proof}
		Since the events in $\hlog \harpoon (\hat{q}(v_1)-1)$ are all non-neighbors of the labels of $T$, the variables in $\mc{v}_{v_1}$ are never used until $\hat{q}(v_1)$. Then, the variables in $\mc{v}_{v_1}$ are resampled with no bearing on the log. This repeats for $v_2, v_3$ and so on. 
	\end{proof}
	
	We need a final bit of notation before we calculate an upper bound for the probability that the $T$-check passes. 
	
	Let $Q = \{S' \in \Omega': \hat{q}_r(S') < \infty\}$. For each stage $s$, let $\tau_s: Q \to \Omega$ send $S' \in Q$ to $S \in \Omega$ such that
	\[
		x_i(S) = x_i^{\wh{v}_s(i)}(S').
	\]
	
	Then, for an event $A \in \mc{A}$, the set $\tau_s^{-1}(A)$ is the event that $A$ is true at stage $s$ of the $T$-check. For each $v \in T$, let $\tau_v = \tau_{\hat{q}(v)}$. Proposition \ref{ConstantVs} helps us bound the $T$-check passing when we condition starting from the \textit{end} of the $T$-check by ensuring that the resample variables do not depend on what happened previously. Since the $T$-check never passes if $\hat{q}(r) = \infty$, we have
	
\begin{align}
	\Pr(\text{The } T\text{-check passes}) &\leq \Pr_{\hat{q}(r)\leq \infty}\left( \bigwedge_{v \in T} \tau^{-1}_v\left([v]\right)\right) \nonumber\\
	&\leq \prod_{r \geq i > 1} \Pr_{\hat{q}_r < \infty} \left(\tau_{v_i}^{-1}\left([v_i]\right) \bigg| \bigwedge_{r\geq j > i} \tau_{v_j}^{-1}\left([v_j]\right) \right)\label{FirstTCheckBound}.
\end{align}

	We will use the $P^*$ condition, Condition~\ref{PStarCondition}, to bound the expression in Line \ref{FirstTCheckBound} factor-wise. However, Condition~\ref{PStarCondition} only applies to probabilities over $(\Omega, \Pr)$. We pull Condition~\ref{PStarCondition} back to $\Omega'$.
	
	\begin{Lemma}
		\label{UsingPStarBound}
		For any event $B \subset \{\hat{q}_r < \infty\}$ such that $B \in \sigma(\mc{x}'\setminus \mc{v}_v)$,
		\[
			\Pr_{\hat{q}(r) < \infty}\left(\tau_v^{-1}([v])\bigg| B\right) \leq P^*([v]). 
		\]
	\end{Lemma}
	\begin{proof}
		Fix $v \in T$. Recall that $\mc{e}_{[v]}$ is the set of all valuations of $\stc{[v]}$. For each $\mu \in \mc{e}_{[v]}$, let $G_\mu$ be the set of all valuations of $\rsp([v])$ that make $[v]$ true whenever $\stc([v])$ is valuated by $\mu$. That is, $\nu \in G_\mu$ if and only if $[v]$ is true when for all $x \in \stc([v])$ and all $y \in \rsp([v])$, we have that $x = \mu(x)$ and $y = \nu(y)$. Then, since $\wh{v}_{\hat{q}(v)}(i)$ is constant for each $x_i \in \rsp([v])$, we get that, for each $\nu \in G_\mu$, 
		
		\[
		\begin{split}
			\tau_v^{-1}(E_\nu)   &= \left\{S \in \Omega': \hat{q}(S)(r) < \infty \text{ and } \wh{R}_{\hat{q}(v)}\harpoon \rsp([v]) = \nu\right\} \\
			&=  \left\{S \in \Omega': \hat{q}(S)(r) < \infty \text{ and } \left(\forall  x_i^j \in \mc{v}_v\right)\left(S(x_i^j) = \nu(x_i)\right) \right\}.
		\end{split}
		\]
		Since $\mc{v}_v$ is a set of mutually independent random variables, we obtain
		\[
			\Pr_{\hat{q}(v) < \infty} \left(\tau_v^{-1}(E_\nu)\right) = \Pr_\Omega (E_\nu).
		\]
		
		The collection of sets $\{\tau^{-1}(E_\mu): \mu \in \mc{e}_{[v]}\}$ is an open cover of $\Omega'$, so we get
		\[
		\begin{split}
			\Pr_{\hat{q}(r) < \infty}\left(\tau_v^{-1}([v])\bigg| B\right) &\leq \sum_{\mu \in \mc{e}_{[v]}} \Pr_{\hat{q}(r) < \infty}\left(\tau_v^{-1}([v])\bigg| B \wedge \tau_v^{-1}(E_\mu)\right)\Pr_{\hat{q}(r) < \infty} (\tau_v^{-1}(E_\mu))\\
			&\leq \sum_{\mu \in \mc{e}_{[v]}}\sum_{\nu \in G_\mu} \Pr_{\hat{q}(r) < \infty}\left(\tau_v^{-1}(E_\nu)\bigg| B \wedge \tau_v^{-1}(E_\mu)\right)\Pr_{\hat{q}(r) < \infty} (\tau_v^{-1}(E_\mu)).
		\end{split}
		\]
		By Claim \ref{IndependentLogs}, we have that $\tau_v^{-1}(E_\mu)$ is independent from $\mc{V}_v$. Since $B$ is also independent from $\mc{V}_v$, and $\tau_v^{-1}(E_\nu) \in \sigma(\mc{v}_v)$, we have that $B \wedge \tau_v^{-1}(E_\mu)$ is independent from $\tau_v^{-1}(E_\nu)$. Thus,  
				\[
		\begin{split}
			\Pr_{\hat{q}(r) < \infty}\left(\tau_v^{-1}([v])\bigg| B\right) &\leq \sum_{\mu \in \mc{e}_{[v]}}\sum_{\nu \in G_\mu} \Pr_{\hat{q}(r) < \infty}\left(\tau_v^{-1}(E_\nu)\right)\Pr_{\hat{q}(r) < \infty} (\tau_v^{-1}(E_\mu))\\
			&= \sum_{\mu \in \mc{e}_{[v]}}\sum_{\nu \in G_\mu} \Pr_{\Omega}\left(E_\nu\right)\Pr_{\hat{q}(r) < \infty} (\tau_v^{-1}(E_\mu))\\
			&= \sum_{\mu \in \mc{e}_{[v]}} \Pr_\Omega([v]| E_\mu) \Pr_{\hat{q}(r) < \infty} (\tau_v^{-1}(E_\mu))\\
			&\leq\sum_{\mu \in \mc{e}_{[v]}} P^*([v]) \Pr_{\hat{q}(r) < \infty} (\tau_v^{-1}(E_\mu))\\
			&= P^*([v]).
		\end{split}
		\]
	\end{proof}
	
	Since the variables in $\mc{v}_v$ are only used to determine whether $[v]$ is true at stage $\hat{q}(v)$ and then resampled, $\bigwedge_{r\geq j > i} \tau_{v_j}^{-1}\left([v_j]\right)$ is independent of $\mc{v}_{v_i}$. Hence, we can use Lemma \ref{UsingPStarBound} to bound Line \ref{FirstTCheckBound}, obtaining
	\[
		\Pr(\text{The } T\text{-check passes}) \leq \prod_{v\in T} P^*([v]). 
	\]
	
	We will apply the following lemma of Moser and Tardos (we use part 2 of the lemma for the computable version).
	
	\begin{Lemma}
		\label{sumMoserTreeBound}
\begin{enumerate}\phantom{f}
	\item \cite{Moser2010}	Suppose there is $z: \mc{A} \to (0,1)$ such that, for each $A \in \mc{A}$,
		\[
			P^*(A) \leq z(A)\prod_{B \in \nbr(A)} \left(1 - z(B)\right).
		\]
		
		Fix $C \in \mc{A}$. Let $\mc{T}_C$ be the set of Moser trees whose root's label is $C$. Then,
		
		\[
			\sum_{T \in \mc{t}_C} \prod_{v \in T} P^*([v]) \leq \frac{z(C)}{1-z(C)}.
		\]	
	\item \cite{Rumyantsev2013}	Furthermore, if there is $\alpha \in (0,1)$ such that, for each $A \in \mc{A}$,
	
		\[
			P^*(A) \leq \alpha z(A)\prod_{B \in \nbr(A)} \left(1 - z(B)\right),
		\]
		then, for every $m \in \mathbb{N}$
		\[
			\sum_{T \in \mc{t}_C, |T| > m}\left( \prod_{v \in T} P^*([v])\right) \leq \alpha^m \frac{z(C)}{1-z(C)}.
		\]
\end{enumerate}
	\end{Lemma}

	We are now ready to complete the proof of Theorem \ref{ALLLL}. Let $\mc{t}_\mc{b}$ be the set of Moser trees whose root's label is a member of $\mc{B}$. By Lemma \ref{sumMoserTreeBound}, we get
	\[
	\begin{split}
		\sum_{T \in \mc{t}_\mc{b}} Pr(T \text{ is generated by an initial} &\text{ segment of the resample algorithm})\\ &\leq \sum_{B \in \mc{B}}\frac{z(B)}{1-z(B)}.
	\end{split}
	\]
	
	Recall that $t_\mc{B}$ is the first step of the resample algorithm at which each $B \in \mc{B}$ is good. Without loss of generality, assume that $\mc{B}$ is $\prec$-downwards closed. Then, each Moser tree produced by a finite initial segment of resample algorithm is a member of $\mc{b}$. Each step of the resample algorithm produces exactly one Moser tree and each Moser tree can be generated by exactly one initial segment of the resample algorithm, so
	\begin{align*}
		\mathbb{E}(t_\mc{b}) &= \mathbb{E}(\text{\# of Moser trees produced by the resample algorithm})\\
		&= \sum_{T \in \mc{t}_\mc{b}} \mathbb{E}\left({\mathbbm{1}}_{T \text{ is generated by an initial segment of the resample algorithm}}\right)\\
		&= \sum_{T \in \mc{t}_\mc{b}} \Pr(T \text{ is generated by an initial segment of the resample algorithm})\\
		&\leq  \sum_{T \in \mc{t}_\mc{b}} \prod_{v \in T} P^*([v])\\
		&\leq \sum_{B \in \mc{B}}\frac{z(B)}{1-z(B)}\\
		&< \infty.
	\end{align*}

	This completes the proof of Theorem \ref{ALLLL}
	
\section{Effective Witnesses}
	
	In this section, we expand upon the proof of Theorem \ref{ALLLL} in order to prove Theorem \ref{CLLLL}.
	
	\subsection{Rewriting Machines and Layerwise Computable Mappings}
	
	To model the resample algorithm of Moser and Tardos, Rumyantsev and Shen introduce a model of probabilistic computation that allows the output tape to be \textit{rewritable} --- that is, the machine can change the contents of each output cell arbitrarily often. We think of these machines as random variables from Cantor space equipped with the Cantor measure to partial functions from $\omega$ to $\{0,1\}$. We can model this behavior with a Turing functional $\Phi$, where $\Phi^B(i,s)$ represents the contents of the $i$'th cell at stage $s$ when run with random source $B$. 
	
	\begin{Definition}[Rewriting Probabilistic Turing Machine]
		\label{def:rpTm}
		A \textbf{rewriting probabilistic Turing machine} is a random variable $\mathbf{\Phi}: 2^\omega \to \{$partial functions from $\omega$ to $\{0,1\}\}$ equipped with a total computable Turing functional $\Phi$ with the property that
		\[
		\mathbf{\Phi}(B)(i) = 
		\begin{cases}
			\lim_{s \to \infty} \Phi^B (i,s) &\text{if the limit exists}\\
			\uparrow &\text{otherwise}
		\end{cases}
		\]
		for all $i$. 
	\end{Definition}
	
	We will be interested in the output distribution of rewriting probabilistic Turing machines $\mathbf{\Phi}$. Suppose that $\Pr(\mathbf{\Phi}(i)\downarrow \text{ for all }i) = 1$.
	Then, the probability distribution $Q$ on the output of $\mathbf{\Phi}$ is determined by its values $Q(\Sigma_x) = \mu(\{B: \mathbf{\Phi}(B) \in \Sigma_x\})$ where $\Sigma_x$ is the cone of infinite extensions of the binary string $x$. We say that $Q$ is computable if $Q(\Sigma_x)$ is uniformly computable with respect to $x$. The following proposition states that we can compute $Q$ if there is a computable function $N(i,\delta)$ such that the probability that the $i$'th entry changes after step $s = N(i,\delta)$ is less than $\delta$.
	
	\begin{Proposition}[\cite{Rumyantsev2013}]
		\label{lwcm}
		Let $\mathbf{\Phi}$ be a rewriting probabilistic Turing machine with Turing functional $\Phi$. Suppose that there is a computable function $N(i,\delta)$ such that the probability that the $i$'th entry changes after step $N(i,\delta)$ is less than $\delta$. Then,
		\begin{enumerate}
			\item For each $i$, $\Pr(\mathbf{\Phi}(i)\downarrow) = 1$.
			\item The output distribution on $\mathbf{\Phi}(i)$ is uniformly computable w.r.t. $i$. 
		\end{enumerate}
	\end{Proposition}
	We call rewriting probabilistic Turing machines satisfying the conclusions of Proposition~\ref{lwcm} \textbf{layerwise computable mappings}.
	\begin{proof}
		It is sufficient to show that, given rational $\delta > 0$ and $x \in 2^{<\omega},$ we can approximate $Q(\Sigma_x)$ with error at most $\delta$. Let $|x| = \text{length}(x)$. Let 
		\[k = \max(\{N(i,\delta/|x|): i < |x|\}).\] 
		
		Because $\Phi$ is a total Turing functional, there is computable function  $m:\omega \to \omega$ such that $\Phi^{\sigma}(i,n)\converges$ for every $n$, $i<n$ and $\sigma \in 2^{m(n)}$. Therefore, for all $n, i\leq n$ and $B \in 2^\omega$, we have that $\Phi^{B|_{m(n)}}(i,n)\converges.$ We approximate $Q(\Sigma_x)$ by 
		\[
		\begin{split} 
			\widehat{Q}(\Sigma_x) &= \frac{\# \{y \in 2^{m(k)}: \Phi^{y}(i,k)\downarrow = x(i) \text{ for all } i < |x|\}}{2^{m(k)}}\\
			&= \Pr(\Phi^{B|_{m(k)}}(i,k)\converges = x(i) \text{ for all }i < |x|). 
		\end{split}
		\]
		Then, 
		\[
		\begin{split}
			|Q(\Sigma_x) - \widehat{Q}(\Sigma_x)| &\leq \Pr\big((\exists i < |x|)(\exists s > k)[\Phi^I(i,s) \neq \Phi^I(i,k)]\big)\\
			&\leq \sum_{i < |x|} N(i,k)\\
			&\leq \delta,        
		\end{split}
		\]
		so we have successfully approximated $Q(\Sigma_x)$ with error less than or equal to $\delta$.
	\end{proof}
	
	Now that we have a computable output distribution for a layerwise computable mapping, we can find a computable element of its output. 
	\begin{Proposition}
		\label{compelmnt}
		Let $\mathbf{\Phi}$ be a layerwise computable mapping and let $F \subset 2^\omega$ be a closed set such that $\Pr(\Phi \in F) = 1$. Then, $F$ has a computable element. 
	\end{Proposition}
	\begin{proof}
		Let $Q(\Sigma_x) = \Pr(x \prec \mathbf{\Phi})$. Part 1 of Proposition \ref{lwcm} implies that $Q(\Sigma_x) = Q(\Sigma_{x^\frown 0}) + Q(\Sigma_{x^\frown 1})$, so if $Q(\Sigma_x) > 0$ then either $Q(\Sigma_{x^\frown 0}) > 0$ or $Q(\Sigma_{x^\frown1}) > 0$. Thus, if $Q(\Sigma_x) > 0$, we can computably define a computable $A \in 2^\omega$ such that $x \prec A$ and for each $\tau \prec A$, $Q(\Sigma_\tau > 0)$. If $A \not\in F$, then, since $F$ is closed, there is some $\tau \prec A$ such that $\tau$ has no extensions in $F$. But $Q(\Sigma_\tau) > 0$, contradicting that $X \in F$ with probability 1. Thus, $A \in F$. 
	\end{proof}
	
	\subsection{Layerwise Computable Resample Algorithm}
	To prove Theorem \ref{CLLLL}, it remains to show that the resample algorithm is a layerwise computable mapping. 
	The following proof contains only superficial changes to the proof in \cite{Rumyantsev2013} that the resample algorithm for the (non-lefthanded) LLL is a layerwise computable mapping. 
	Let $\Phi^S(i,s)$ be the value of $x_i$ at stage $s$ of the resample algorithm with random source $B \in \Omega'$. We need to show that the rewriting probabilistic Turing machine $\mathbf{\Phi}$ associated with $\Phi$ is a layerwise computable mapping. We also need to show that none of the bad events $A \in \mathcal{A}$ are true under the valuation of the $x \in \mathcal{X}$ given by $\mathbf{\Phi}(S)$ for almost every $S \in \Omega'$. Then, we can apply Proposition \ref{compelmnt} to the set $F$ of valuations of the $x \in \mathcal{X}$ that make each $A \in \mathcal{A}$ false to show that $F$ has a computable element.
	
	To show that $\mathbf{\Phi}$ is layerwise computable, we need to compute $N(i,\delta)$ such that $x_i$ changes with probability less than $\delta$ after stage $N(i,\delta)$. Because each $x_i$ is involved in finitely many events and the set of events is uniformly computable, it is sufficient to, for each $A \in \mc{A}$ such that $x_i \in \rsp(A)$, find an $M(A,\delta)$ such that the probability that $A$ is resampled after stage $M(A,\delta)$ is less than $\delta$. Fix $A$, and $\delta$. Let $\mc{B} \subset \mc{A}$ be a finite, $\prec$-downward closed set of events such that $\mc{B}$ contains all events of distance $m$ or less from $A$ in the graph $(\mc{A},\nbr)$ for $m$ such that 
	
	\[
	\alpha^{m} \frac{z(A)}{1 - z(A)} \leq \delta/2.
	\] 
	
	Recall that the stopping time $\tau_\mc{b}$ is the first stage such that each $B \in \mc{B}$ is good. The event ``$A$ is resampled after time $t$'' is covered by the events $t < \tau_\mc{b}$ and $t \geq \tau_\mc{b}$. We show that the probability of both of these events effectively converges to 0. 
	
	By Theorem \ref{ALLLL}, we have 
	\[
		\mathbb{E}(\tau_\mc{b}) \leq \sum_{B \in \mc{b}}(1-z(B)). 
	\]
	
	By Markov's inequality, 
	
	\[
	\Pr(\tau_\mc{b} \geq t) \leq \frac{\mathbb{E}(\tau_\mc{b})}{t},
	\]
	which we can solve to find $s$ large enough so that 
	\begin{equation}
		\label{eqn:t1Con}
		\Pr(\tau_\mc{b} \geq s) < \delta/2.
	\end{equation}
	Note that such an $s$ is uniformly computable from $A$, $i$ and $\delta$. We set $M(A,\delta) = s$.	Since $\Pr(\tau_\mc{b} \geq s) < \delta/2$, we also have that 
	\[
	\Pr(A \text{ is resampled after stage } s \text{ and } \tau_\mc{b} \geq s) < \delta/2.
	\]
	
	Now we bound the probability that  $\tau_\mc{b} < s$ and $A$ is resampled at some stage $t_1 >  s > \tau_\mc{B}$. Consider the log of resampling $E_1,E_2,\dots$. Suppose that $A$ is resampled at some stage $t_1 > s > \tau_{\mc{b}}$.  We claim that there must be a sequence of stages $\tau_\mc{b} < t_m < t_{m-1} \dots < t_3< t_2 < t_1$ such that for each $1 \leq j < m$, we have that $E_{t_{j+1}} \in \nbr^+(E_{j})$. 
	
	We prove the claim by induction on the length of the sequence. For the base case, we have that $A \neq E_{\tau_\mc{b}}$ because $A$ is good at stage $\tau_\mc{b}$. Let $n < m$ and suppose there is $\tau_\mc{B}  < t_n < \dots < t_2 < t_1$  such that for each $1 \leq j < n$ we have that $E_{t_{j+1}} \in \nbr^+(E_{t_j}).$ Then, the distance between $E_{t_n}$ and $A$ is at most $n$, so $E_{t_n} \in \mc{B}$. Since $\mc{B}$ is $\prec$-downward closed, $E_{\tau_\mc{b}} \not \in \mc{b}$, we have that $E_{\tau_\mc{b}} \succ E_{t_n}$. Then, by Proposition \ref{LoseProgressAlongNeighbors}, there is $k$ with $\tau_\mc{b} < k < t_n$ such that $E_k \in \nbr^+(E_{t_n})$. Set $t_{n+1} = k$. This completes the proof of the claim.

	Let $X$ be the event that an initial segment of the resample algorithm produces a tree of size at least $m$ with root node labeled $A$. Consider the Moser tree $T$ associated with $E_1,\dots, E_{t_1} = A$. The subsequence $E_{t_m},E_{t_{m-1}},\dots, E_{t_1}$ is a chain of neighbors, so $T$ has at least $m$ vertices and root node labeled $A$. Hence, we have shown if $A$ is resampled after stage $\tau_\mc{b}$, then $X$ is true.  By the second part of Lemma \ref{sumMoserTreeBound}, it follows that
	\begin{align*}
		\Pr(A\text{ is resampled at stage } t_1 > s \text{ and } \tau_\mc{b} < t_1) &\leq \Pr(X)\\
		&\leq \sum_{T \in \mathcal{T}_{A},|T|\geq m } \left(\prod_{v \in T} P^*([v])\right)\\
		&\leq\alpha^{m} \frac{z(A_i)}{1-z(A_i)}\\
		&\leq\delta/2.
	\end{align*}
	 
	 Thus, we have that $\Pr(A \text{ is resampled after stage }s) \leq \delta$. Therefore, the resample algorithm is a layer-wise computable mapping. To apply Proposition~\ref{compelmnt}, we just need to check that the output sequence 
	 \[
	 \lim\limits_{t \to \infty} x_1^{V_t(x_1)}\lim\limits_{t \to \infty} x_2^{V_t(x_2)}\lim\limits_{t \to \infty} x_3^{V_t(x_3)}\dots
	 \]
	 makes each $A \in \mathcal{A}$ false almost surely, and that making each $A \in \mathcal{A}$ false is closed. To see that latter, note that each $A$ being true is open, so the intersection of them all being false is closed. To see the former, fix $A \in \mathcal{A}$ and suppose that $A$ is true in the output sequence. Then, the resample algorithm resamples some $B$ with $B \prec A$ infinitely many times. However, this happens with probability $0$, as we have just shown that as $t$ goes to infinity, the probability that $B$ gets resampled after time $t$ goes to $0$. Thus, the probability that \textit{any} of the countably many $A \in \mathcal{A}$ is true in the output sequence is also $0$. 
	 
	 Thus, the resample algorithm is a layerwise computable mapping which almost surely converges in the closed set $F$ of all valuations of the $x_i \in \mathcal{X}$ that make each $A \in\mathcal{A}$ false. By Proposition \ref{compelmnt}, $F$ has a computable element. This completes the proof of Theorem \ref{CLLLL}.
	
\section{Applications To Non-Repetitive Sequences}
	The study of non-repetitive sequences is said to have originated with Thue's proof that there exists an infinite ternary sequence without any adjacent identical blocks \cite{Thue1906, Thue1912}. Note that this fails spectacularly for binary sequences; when attempting to construct such a sequence, one will immediately conclude that the task is impossible upon trying to extend either $010$ or $101$. Finding sequences with various degrees and types of non-repetitiveness has since become an active area of research (see \cite{Gryctzuk2008} for an overview). In this section, we constructivise and effectivize several generalizations of Thue's theorem whose proofs were previously non-constructive. We begin with our applications to binary sequences. 

	\subsection{Making Long Identical Blocks Far Apart}
	One direction that Thue's result can be generalized to binary sequences is by requiring that long identical intervals be far apart. The following is one of the first applications of the infinite version of the Lov\'asz local lemma, which we restate for convenience.

	\begin{Theorem}[\citeauthor{Beck1981}\cite{Beck1981}]
		\label{thm:Beck}
		Given arbitrary small $\epsilon > 0$, there is some $N_\epsilon$ and an infinite $\{0,1\}$-valued sequence $a_1,a_2, a_3,\dots$ such that any two identical intervals $a_k,\dots,a_{k+n-1}$ and $a_\ell,\dots,a_{\ell + n -1}$ of length $n > N_\epsilon$ have distance $\ell - k$ greater than $(2-\epsilon)^n$.
	\end{Theorem}
	A game version of Theorem \ref{thm:Beck} is studied in \cite{Pegden2011}, which we can partially effectivize. We defer discussion of game versions to section \ref{section:BinarySequenceGames}.

	As discussed in the introduction, the computable local lemma of Rumyantsev and Shen does not readily yield a computable version of Theorem~\ref{thm:Beck} due to the events $A_{k,\ell,n}$ having infinitely many neighbors. Theorem \ref{CLLLL} shrinks the neighborhood sets and avoids this problem. 
	
	\begin{theorem}[Computable Version of Theorem \ref{thm:Beck}]
		\label{computableBeck}
		There is computable Turing functional $\Phi$ such that, given arbitrary small $\epsilon > 0$, there is some $N_\epsilon$ such that $\Phi(\epsilon)$ is a $\{0,1\}$-valued sequence in which any two identical intervals of length $n > N_\epsilon$ have distance greater than $f(n) = (2-\epsilon)^n$.
	\end{theorem}
	
	\begin{proof}
		Fix $\epsilon > 0$ and $N > 0$. We will show that if $N$ is large enough, then $N_\epsilon = N$ witnesses the theorem to be proved. 
		We set up the application of Theorem \ref{CLLLL}. Let $\mathcal{X} = \{x_1,x_2,...\}$ be the bits in our binary sequence. Recall that $A_{k,l,n}$ is the event that $x[k,k+n) = x[l,l+n)$. Note that $A_{k,l,n}$ being false implies that $A_{k,l,m}$ is false for each $m > n$. It is therefore enough to only consider $A_{k,l,n}$ such that $n$ is minimal for $l-k$, that is, such that $l-k = \lceil f(n) \rceil$. Let $\mathcal{A} = \{A_{k,l,n}: l - k = \lceil f(n) \rceil \text{ and } n > N\}$. Let $\rsp(A_{k,l,n}) = [l,l+n)$. Then, $\stc(A_{k,l,n}) = x[k,k+n) \setminus x[l,l+n)$. Note that for each $x_m$, we have that $x_m \in \rsp(A_{k,\ell,n}$ for finitely many $A_{k,\ell,n} \in \mc{A}$, since there are only finitely many $k < \ell$. 
		
		The computability conditions are immediate and the requirement for a linear order follows from Lemma~\ref{IntervalsWork}.
		
		We check the probabilistic conditions. Let $P^*(A_{k,l,n}) = 2^{-n}$. We need to show that for any valuation $\mu$ of $\stc(A_{k,l,n})$, 
		\[
		P^*(A_{k,l,n}) \geq \Pr(A_{k,l,n}|E_\mu).
		\]
		In fact, $P^*(A_{k,l,n}) = \Pr(A_{k,l,n}|E_\mu)$. This is immediate if $\rsp(A_{k,l,n}) \cap x[k,k+n)  = \emptyset$. If $\rsp(A_{k,l,n}) \cap x[k,k+n) \neq \emptyset$, suppose the entries of $x[k,l)$ are given by valuation $\mu$ and $l = k + j$. Then for $A_{k,l,n}$ to be true, $x_{l+i}$ must equal $\mu(x_{k + i})$ for all $i \leq j$ as in the case where the compared intervals don't overlap. Then, when we consider the first non-overlap $x_{l + j}$, still do not have a choice, because $x_{l +j}$ must be equal to $x_l$, whose value must be $\mu(x_{k})$. Thus, there is only one valuation of the variables in $\rsp(A_{k,l,n})$ under which $A_{k,l,n}$ is true when the variables in $\stc(A_{k,l,n})$ have values given by $\mu$. So, $\Pr(A_{k,l,n}|E_\mu) = 2^{-n} = P^*(A_{k,l,n})$. 
		
		Finally, we show that the main local lemma condition holds. Let $z(A_{k,l,n}) = \frac{1}{f(n)n^3}$. We need to show that there is $\alpha \in (0,1)$, such that each $A_{k_0,l_0,n_0} \in \mathcal{A}$,
		\[
		P^*(A_{k_0,l_0,n_0}) = 2^{-n_0} \leq \alpha z(A_{k_0,l_0,n_0}) \prod_{A_{k,l,n} \in \nbr(A_{k_0,l_0,n_0})} (1 - z(A_{k_0,l_0,n_0})).
		\]
		
		For each $k_0,l_0,n_0$ there are at most $(n+n_0)$ many intervals of length $n$ that have non-empty intersection with $x[l_0,l_0+n_0)$. 
		For any given interval $I$ and fixed $\ell$, we have that $I = [l,l + n)$ for at most $f(n)$ many pairs $(k,l)$ such that $l - k \leq f(n)$. Together with the previous observation, we conclude that  $\{(k,l): A_{k,l,n} \in \nbr(A_{k_0,l_0,n_0})\}$ is of size at most $(n + n_0)f(n)$ for each $n$. 
		This, along with the definition of $z(A_{k,l,n})$ and the fact that $0 < 1 - z(A_{k,l,n}) < 1$ for all $k,l,n$, gives us that
		
		\begin{multline*}
			z(A_{k_0,l_0,n_0}) \prod_{A_{k,l,n} \in \nbr(A_{k_0,l_0,n_0})} (1 - z(A_{k,l,n})) \geq \\
			\frac{1}{f(n_0)n_0^3} \prod_{n \geq N} \left(1 - \frac{1}{f(n)n^3}\right)^{(n_0 + n)(f(n))}. 
		\end{multline*}
		Since $(1-\frac{a}{x})^x \geq (1-a)$ for $0 < a < 1$ and $x \geq 1$, the right hand side is greater than or equal to 
		\begin{align}
			&\frac{1}{f(n_0)n_0^3} \prod_{n \geq N} \left(1 - \frac{1}{n^3}\right)^{n_0} \left(1 - \frac{1}{n^2}\right)^{4} \notag \\
			= &\frac{1}{f(n_0)n_0^3} \left( \prod_{n\geq N} \left(1 - \frac{1}{n^3}\right)\right)^{n_0}  \prod_{n \geq N} \left(1 - \frac{1}{n^2}\right). \label{eq:SumToProdLine} 
		\end{align}
		For $0 < a_n < 1$, $\prod_{n \geq N}(1-a_n) \geq 1 - \sum_{n \geq N} a_n$. So, Line \ref{eq:SumToProdLine} is greater than or equal to
		\begin{equation}
			\label{eq:BeckPfSumProduct}
			\frac{1}{f(n_0)n_0^3}\left(1-\sum_{n \geq N}\frac{1}{n^3}\right)^{n_0}\left(1-\sum_{n \geq N}\frac{1}{n^2}\right).
		\end{equation}

		For $N\geq 2$, the two sums in Line \ref{eq:BeckPfSumProduct} are bounded above by $\frac{1}{N-1}$, so Line \ref{eq:BeckPfSumProduct} is greater than or equal to
		\[
		\frac{(2-\epsilon)^{-n_0}}{n_0^3}\left(1 - \frac{1}{N-1}\right)^{n_0 + 1}.
		\]
		If $N$ is large enough, this is greater than or equal to $2P^*(A_{k_0,l_0,n_0}) = 2^{-n_0 + 1}$, fulfilling the local lemma condition with $\alpha = \frac{1}{2}$.
	\end{proof}
	\subsection{Making Adjacent Blocks Very Different}
	Another way that Thue's result can be generalized to binary sequences is by requiring that long adjacent blocks be as different as possible. The following result appears as an exercise in \cite{Alon1992}.
	
	\begin{Theorem}[\cite{Alon1992}]
		\label{thm:Alon}
		Given arbitrary small $\epsilon >0$, there is some $N_\epsilon$ and an infinite $\{0,1\}$-valued sequence $a_1,a_2,a_3,...$ such that any two adjacent intervals of length $n > N_\epsilon$ differ in at least $(\frac{1}{2}-\epsilon)n$ many places. That is, for each $k$ and $n > N_\epsilon$, $a_{k + i} \neq a_{k + n + i}$ for at least $(\frac{1}{2}-\epsilon)n$ many $i$ with $0 \leq i < n$.  
	\end{Theorem}
	
	Pegden also studies game versions of Theorem~\ref{thm:Alon}. We discuss this in Section~\ref{section:BinarySequenceGames}.
	
	Theorem~\ref{CLLLL} also effectivize Theorem~\ref{thm:Alon}. 
	
	\begin{Theorem}[Computable Version of Theorem \ref{thm:Alon}]
		\label{computableAlon}
		Given arbitrary small $\epsilon >0$ there is some $N_\epsilon$ and a computable $\{0,1\}$-valued sequence $a_1,a_2,a_3,...$ such that any two adjacent intervals of length $n > N_\epsilon$ differ in at least $(\frac{1}{2}-\epsilon)n$ many places; that is, for each $k$ and $n > N_\epsilon$, $a_{k + i} \neq a_{k + n + i}$ for at least $(\frac{1}{2}-\epsilon)n$ many $i$ with $0 \leq i < n$.  
	\end{Theorem}
	\begin{proof} 
		Fix $\epsilon > 0$ and $N > 0$. We will show that if $N$ is large enough, then $N_\epsilon = N$ witnesses the theorem. 
		We again apply the lefthanded computable  Lov\'asz local lemma.  Let $\mathcal{X} = \{x_1,x_2,...\}$ be the bits in our binary sequence. Let $A_{k,n}$ be the event that blocks $x[k,k+n)$ and $x[k+n,k+2n)$ share at least $(\frac{1}{2} + \epsilon)n$ many entries. Let $\mathcal{A} = \{A_{k,n}: n > N\}$ and $\rsp(A_{k,n}) = x[k+n,k+2n)$. Then, $\stc(A_{k,n}) = x[k+n,k+2n)$.
		
		The computability conditions are immediate and the requirement for a linear order follows from Lemma~\ref{IntervalsWork}.
		
		We check the probabilistic conditions. The probability that $\rsp(A_{i,n})$ shares \textit{exactly} k entries with $\stc(A_{i,n})$ is ${n \choose k} 2^{-k}$, so we set
		\[
		P^*(A_{i,n}) = \Pr(A_{i,n}) = 2^{-n} \sum_{r = \lceil (\frac{1}{2} + \epsilon)n  \rceil}^n {n \choose r}. 
		\]
		We need to show that for any valuation $\mu$ of $\stc(A_{i,n})$, 
		\[
		P^*(A_{k,l,n}) \geq \Pr(A_{i,n}|E_\mu).
		\]
		
		We have that $\rsp(A_{k,n}) \cap x[k,n+k) = \emptyset$ for all $k$ and $n$, so $\Pr(A_{i,n}|E_\mu) = \Pr(A_{i,n})$.  Since the greatest of the ${n \choose r}$ is ${n \choose \lceil(\frac{1}{2}+\epsilon)n\rceil}$, 
		\[
		\Pr(A_{i,b}) \leq n 2^{-n} {n \choose \lceil(\frac{1}{2}+\epsilon)n\rceil}.
		\]
		
		It is known that there is an $\alpha < 1$ such that ${N \choose \lceil(\frac{1}{2} + \epsilon)N \rceil} < (\alpha 2)^N$ for large enough $N$, so there is $N_\epsilon$ such that for $n > N_\epsilon$,
		\[
		\Pr(A_{i,n}) < n 2^{-n} (\alpha 2)^n = n\alpha^n,
		\]
		which has limit $0$ as $n \to \infty$.
		
		Finally, it remains to show that the local lemma condition holds. Let $z_{i,n} = \frac{b^n}{n}$ for some $\alpha^{1/2} < b < 1$. Fix $A_{i_0, n_0}$. For each $n$, the neighborhood $A_{k_0,n_0}$ has at most $n_0 + n$ elements $B$ with $|\rsp(B) | = n$, so it suffices for the Lovasz local lemma to check that 
		\[
		\frac{1}{2}b^n \prod_{n = N_\epsilon}^{\infty}(1 - (b/n)^n)^{n + n_0} \geq n_0\alpha^{n_0}. 
		\]
		for large enough $N_\epsilon$, with $\frac{1}{2}$ being the $\alpha < 1$ from the computable lefthanded Lov\'asz local lemma. 
		
		The left-hand side is equal to 
		
		\[\begin{split}
			\frac{b^{n_0}}{n_0}\left(\prod_{n = N_\epsilon}^\infty \left(1 - \frac{b^{n}}{n}\right)^{2n_0}\right) \left(\prod_{n = N_\epsilon}^\infty \left(1 - \frac{b^n}{n}\right)^{2n}\right) &\geq \frac{b^{n_0}}{2n_0}\left(\prod_{n = N_\epsilon}^\infty \left(1 - b^n\right)^{n_0}\right) \left(\prod_{n = N_\epsilon}^\infty \left(1 - b^n\right)\right)\\
			&\geq \frac{b^{n_0}}{2n_0}\prod_{n = N_\epsilon}^\infty \left(1 - b^n\right)^{n_0 + 1}\\
			&\geq \frac{b^{n_0}}{2n_0} \left(1 - \sum_{n = N_\epsilon}^\infty b^n \right)^{n_0 + 1}\\
			&=\frac{b^{n_0}}{2n_0}\left(1 - \frac{b^{N_\epsilon}}{1-b}\right)^{n_0+1}\\
			&\geq \frac{b^{n_0}}{2n_0} (b^{n_0 + 1})\\
			&= \frac{ b^{2n_0+1}}{2n_0}\\
			&\geq n_0\alpha^{n_0},
		\end{split}
		\]

		where the last three lines are for $n_0$ large enough (adjust $N_\epsilon$ accordingly). 
	\end{proof}

	\subsection{Defeating Strategies in Binary Sequence Games}
	\label{section:BinarySequenceGames}
	
	Although we obtain computable versions of Theorems~\ref{thm:Beck} and \ref{thm:Alon} from Theorem~\ref{CLLLL}, their non-constructive versions do not require the lefthanded local lemma. Instead, Pegden uses the lefthanded local lemma to analyse \textit{game versions} of these theorems. 
	
	 Consider a game called the \textit{binary sequence game} in which two Players take turns selecting bits in a binary sequence. Player 1 picks the odd bits and Player 2 picks the even bits. The binary sequence game generates the sequence
	\[
	a_1a_2a_3 a_4 a_5 a_6 \dots = e_1 d_2 e_3 d_4 e_5 d_6 \dots
	\]
	where $e_{2n+1}$ is chosen by Player 1 with knowledge of all preceding bits in the sequence (but not future bits) and Player 2 chooses $d_{2n}$ similarly. Pegden found that Player 1 has a strategy to ensure that binary sequence game produces a non-repetitive sequence, as stated below.
	
	\begin{theorem}[\cite{Pegden2011}]
		\label{thm:Beck2}
		For every $\epsilon > 0$ there is an $N_\epsilon$ such that Player 1 has a strategy in the binary sequence game ensuring that any two identical blocks $x[i,i + n) = x_i x_{i+1} x_{i+2}\dots x_{i+n-1}$ and $x[j,j + n) = x_j x_{j+1} x_{j+2}\dots x_{j+n-1}$ of length $n > N_\epsilon$ have distance at least $f(n) = (2-\epsilon)^{n/2}$. That is, if $x_{i+s} = x_{j+s}$ for all $0\leq s <n$ and $n > N_\epsilon$, then $|i -j| < (2-\epsilon)^{n/2}$.  
	\end{theorem}
	
	\begin{theorem}[\cite{Pegden2011}]
		\label{thm:Alon2}
		For every $\epsilon > 0$ there is an $N_\epsilon$ such that Player 1 has a strategy in the binary sequence game ensuring that any two adjacent blocks $x[i,i + n)$ and $x[i+n, i + 2n)$ of length $n > N_\epsilon$ differ in at least $n\left(\frac{1}{4}-\epsilon\right)$ places. That is, if $n > N_\epsilon$ then $x_{i+s} \neq x_{i+n+s}$ for at least $n\left(\frac{1}{4}-\epsilon\right)$ many $0 \leq s < n$. 
	\end{theorem}
	
	Unlike Theorems \ref{thm:Beck} and \ref{thm:Alon}, even the finite versions of these theorems cannot be proven using the classical local lemma. This is because Player two can change their moves depending on what Player one does. This means that, when setting up the local lemma to find a strategy for Player 1, each event depends on every move that came before it. For example in the context of Theorem $\ref{thm:Beck2}$, the event $A_{k,l,n}$ that the intervals $x[i,i+n)]$ and $x[l,l+n)$ are identical depends not only on the moves Player 1 makes in those intervals, but on every move Player 1 has made beforehand. This explodes the size of the neighborhood relation. To prove Theorems \ref{thm:Beck2} and \ref{thm:Alon2}, Pegden introduces an extension of the LLL called the lefthanded local lemma. Pegden uses the lefthanded local lemma to prove the following finite version of Theorem \ref{thm:Beck2}. 
	
	\begin{Proposition}
		\label{Yeach}
		For every $\epsilon > 0$ there is an $N_\epsilon$ such that, for each $M>0$ and Player 2 strategy $g$ in the binary sequence game of length $M$ length, there is a sequence $e_1e_3e_5...e_{M}$ of Player 1 moves that when played against $\Phi$, any two identical blocks of length $n > N_\epsilon$  in the resulting sequence have distance at least $(2-\epsilon)^{n/2}$.  
	\end{Proposition}
	The full Theorem \ref{thm:Beck2} results from the following compactness argument. The proof of Theorem \ref{thm:Alon2} follows the same pattern. 
	\begin{proof}[Proof of Theorem~\ref{thm:Beck2}]
		By open determinacy, it is sufficient to show that there is no winning strategy for Player 2. Fix Player 2 strategy $g$. Fix game length $M>0$. By Proposition~\ref{Yeach}, there is a string of Player 1 moves such that $g$ does not win within $M$ moves. Let $T$ be the tree of such strings of Player 1 moves. $T$ is a finitely branching tree, so by K\"{o}nig's lemma, there is an infinite path through $T$, so $g$ is not a winning strategy for Player 2. Since Player 2 has no winning strategy and they are playing an open game, Player 1 has a winning strategy by open determinacy. 
	\end{proof}
	
	This proof is non-constructive on account of the use of both K\"{o}nig's lemma and open determinacy. By using Theorem \ref{CLLLL}, we can effectivize this use of K\"{o}nig's lemma: we will show that $T$ has a $g$-computable path. However, finding whether the determinacy argument can be effectivised may require different methods and is outside the scope of the present investigation. 
	
	To set up the local lemma in the proofs of the two following theorems, fix Player 2 strategy $g$. Let the probability space be $2^{2\N+1} = \{e = e_1,e_{3},e_{5},...\}$ with $\{x = x_1,x_{3},x_{5},...\}$ as a random variable $x: 2^{2\N+1} \to \{0,1\}$ such that $x_{2i+1}(e) = e_{2i+1}$. Then, we can interpret Player 2's moves $d_2,d_4,d_6,...$ as functions $2^{2\N+1}\to \{0,1\}$ with $d_{2i}(e) = g(e_1,e_{3},...,e_{2i-1})$. Likewise, our sequence $ a(x) = a_1(x),a_2(x),a_3(x),...$ is a sequence of random variables with 
	\[
	a_k(e) = \begin{cases}
		x_k(e) &\text{ if } k = 2n+1 \text{ for some }n\in\N \\
		d_k(e) &\text{ otherwise}\\
	\end{cases}.
	\] 
	
	\begin{theorem}
		\label{BeckComputableGame}
		For every $\epsilon > 0$ there is an $N_\epsilon$ such that for each Player 2 strategy $g$ in the binary sequence game, there is a $g$-computable sequence $e_1e_3e_5...$ of Player 1 moves that when played against $g$, any two identical blocks of length $n > N_\epsilon$  in the resulting sequence have distance at least $f(n) = (2-\epsilon)^{n/2}$.
	\end{theorem}
	\begin{proof}
		Fix $\epsilon > 0$ and $N > 0$. We will show that if $N$ is large enough, then $N_\epsilon = N$ witnesses the theorem. Define $A_{k,l,n}$ as before. Since all of Player 1s previous moves can effect the truth of $A_{k,l,n}$, $\vbl(A_{k,l,n}) = x[0,l+n) \cap \mathcal{X} = \{x_{2i+1}: 2i+1 < l + n\}$. Set $\rsp(A_{k,l,n}) = x[l,l+n) \cap \mathcal{X}$. Let $\mathcal{A} = \{A_{k,l,n}: l - k = \lceil f(n) \rceil \text{ and } n > N\}$. For the same reason as in the proof for the non-game version, avoiding each $A_{k,l,n} \in \mathcal{A}$ is sufficient. Although we do not need to, we will assume that $N$ is large enough so that  $f(n) > n$ for all $n > N$. Then, $l - k > n$ for each $A_{k,l,n} \in \mathcal{A}$, so $x[k,k+n) \cap x[l,l+n) = \emptyset$ for each $A_{k,l,n} \in \mathcal{A}$. For each $x_{2i+1}$, there are fewer than $i^2$ many elements in $\{A_{k,l,n} \in \mathcal{A}:x_{2i+1} \in \rsp(A_{k,l,n})\} \subset \{A_{k,l,\lceil f(l-k) \rceil}: k < l \leq 2i+1\}$, so the set of events that have $x_{2i+1}$ in their resample sets is finite. The other computability conditions are immediate and Lemma~\ref{IntervalsWork} supplies the linear order. 

		To check the probabilistic conditions, we start with the $P^*$ condition. Fix $A_{k,l,n} \in \mathcal{A}$. Our guiding principle is to pick $P^*(A_{k,l,n})$ as small as possible such that, for each valuation $\mu$ of $\stc(A_{k,l,n})$, $\Pr(A_{k,l,n} \mid E_\mu) \leq P^*(A_{k,l,n})$. Player 1 has control of half of the bits, and it turns out that $P^*(A_{k,l,n}) = 2^{-(n-1)/(2)}$ is satisfactory. To see this, fix $\mu$ that is a valuation of $\stc(A_{k,l,n})$. Then, since $\stc(A_{k,l,n})$ is an initial segment of Player 1 moves, $\mu$ fixes the initial segment $a_1,a_2,...,a_{\min(\rsp(A_{k,l,n}))-1}$. Thus, for each $S,S' \in E_\mu$ and $i \leq \max(\stc(A_{k,l,n}))$, we have that $a_i(S) = a_i(S')$. Let $\mu(a_i)$ be this constant $a_i(S)$. Let $E_{\mu,A_{k,l,n}}$ be the event that for each $x_{2i+1} \in \rsp(A_{k,l,n})$,  $x_{2i+1} = \mu(a_{2i+1 -(l - k)})$. Then $(A_{k,l,n} \cap E_\mu) \subset E_{\mu,A_{k,l,n}}$. Trivially, $(A_{k,l,n} \cap E_\mu) \subset E_\mu$, so
		\[
		(A_{k,l,n} \cap E_\mu) \subset (E_{\mu,A_{k,l,n}} \cap E_\mu).
		\]
		Since $E_\mu$ and $E_{\mu,A_{k,l,n}}$ are statements about the of values of disjoint sets of independent random variables, they are independent. Therefore,
		\[
		\begin{split}
			\Pr(A_{k,l,n}|E_\mu) &= \frac{\Pr(A_{k,l,n} \cap E_\mu)}{Pr(E_\mu)}\\
			&\leq \frac{\Pr(E_{\mu,A_{k,l,n}} \cap E_\mu)}{\Pr(E_\mu)}\\
			&= \Pr(E_{\mu,A_{k,l,n}}).
		\end{split}
		\]
		$|\rsp(A_{k,l,n})| \geq (n-1)/(2)$, so $\Pr(A_{k,l,n}|E_\mu) \leq \Pr(E_{\mu,A_{k,l,n}}) \leq 2^{-(n-1)/(2)}$, as desired. 
		
		Next we are left to show that there is $z: \mathcal{A} \to (0,1)$ such that 
		\begin{equation}
			\label{eq:gameBeckLLCon}
			P^*(A_{k,l,n}) \leq z(A_{k,l,n})\prod_{B \in \nbr(A_{k,l,n})} (1-z(B)).
		\end{equation}
		The proof of Condition \ref{eq:gameBeckLLCon} is almost identical to the corresponding part of the proof of Theorem~\ref{computableBeck}, with the definition of $f(n)$ updated.  
	\end{proof}

	\begin{theorem}
		\label{AlonComputableGame}
		For every $\epsilon > 0$ there is an $N_\epsilon$ such that for each Player 2 strategy $g$ in the binary sequence game, there is a $g$-computable sequence $e_1e_3e_5...$ of Player 1 moves that when played against $g$, any two adjacent  blocks of length $n > N_\epsilon$  in the resulting sequence differ in at least $n(\frac{1}{4}-\epsilon)$ many entries. 
	\end{theorem}
	\begin{proof}
		Fix $\epsilon > 0$ and $N > 0$. We will show that if $N$ is large enough, then $N_\epsilon = N$ witnesses the theorem. For each $k,n \in \mathbb{N}$ with $k < l$, let $A_{k,n}$ be the event that  blocks $a_k,a_{k+1},...,a_{k+n-1}$ and $a_{k+n},a_{k+n+1},...,a_{k+2n-1}$ share at least $(\frac{3}{4} + \epsilon)n$ many entries. Then, $\vbl(A_{k,n}) = (x[0, k+2n) \cap \mathcal{X}) = \{x_{2i+1}: k \leq 2i+1 \leq k+2n-1 \}$. Set $\rsp(A_{k,n}) = x[k+n,k+2n) \cap \mathcal{X}$. Let $\mathcal{A} = \{A_{k,n}: n > N\}$. For each $x_{2i+1}$, there are fewer than $(2i+1)^2$ many elements in $\{A_{k,n} \in \mathcal{A}:x_{2i+1} \in \rsp(A_{k,n})\} \subset \{A_{k,n}: k + n \leq 2i+1\}$, so the set of events that have $x_{2i+1}$ in their resample sets is finite.  The other computability conditions are immediate and Lemma~\ref{IntervalsWork} supplies the linear order.
		
		To check the probabilistic conditions, we start with the $P^*$ condition. Fix $A_{k,n} \in \mathcal{A}$. Our guiding principle is to pick $P^*(A_{k,l,n})$ as small as possible such that, for each valuation $\mu$ of $\stc(A_{k,n})$, $\Pr(A_{k,n} \mid E_\mu) \leq P^*(A_{k,n})$. Player 1 has control of half of the bits, and we cannot rely on Player 2 not matching any bits, so we choose
		\[
		P^*(A_{k,n}) = \frac{1}{2^{\lfloor n/2 \rfloor}} \sum_{r=\lceil (\frac{1}{4} + \epsilon)n \rceil}^{\lfloor n/2 \rfloor} {\lfloor n/2 \rfloor \choose r}. 
		\]
		To see this suffices, fix $\mu$ that is a valuation $\stc(A_{k,n})$. Then, since $\stc(A_{k,n})$ is an initial segment of Player 1 moves, $\mu$ fixes the initial segment $a_1,a_2,...,a_{k+n-1}$. Thus, for each $S,S' \in E_\mu$ and $i \leq (k+n-1)$, we have that $a_i(S) = a_i(S')$. Let $\mu(a_i)$ be this constant $a_i(S)$.
		Fix $S \in E_\mu$. Then, $a_i(s) = \mu(a_i)$ for each $i \leq k + n-1$. If $S \in A_{k,n}$, then it is necessary that Player 1s moves in $[k+n,k+2n)$ match at least $\lfloor (\frac{1}{4} + \epsilon)n \rfloor$ many bits on their turn, even if Player 2 always matches the corresponding bit each of their at most $\lceil n/2 \rceil$ moves. $P^*(A_{k,n})$ is the probability that Player 1 matches at least $\lfloor (\frac{1}{4} + \epsilon)n \rfloor$ bits, which is required for $A_{k,n}$ to be true. 
		
		Next we are left to show that there is $z: \mathcal{A} \to (0,1)$ such that 
		\begin{equation}
			\label{eq:gameAlonLLCon}
			P^*(A_{k,l,n}) \leq z(A_{k,l,n})\prod_{B \in \nbr(A_{k,l,n})} (1-z(B)).
		\end{equation}
		The proof of Condition \ref{eq:gameAlonLLCon} is almost identical to the corresponding part of the proof of theorem \ref{computableBeck}.
	\end{proof}
	
	Theorems~\ref{BeckComputableGame} and \ref{AlonComputableGame} show that for each Player 2 strategy $g$, there is a $g$-computable Player 1 sequence which defeats it. It is natural to ask whether there is a Player 1 sequence which defeats \textit{all} Player 2 strategies. However, this is not the case. 
	
	\begin{Proposition}
		\label{prop:noUniform}
		Let $\{e_{2i+1}^k\}_{i,k\in\omega}$ be a $\{0,1\}$-valued matrix. Then, for any odd $N$, there is a sequence  $\{d_{2i}\}_{i \in \omega} \leq_T \{e_{2i+1}^k\}_{i,k\in\omega}$ such that, for any $k$, the sequence
		\[
		a_1a_2a_3...=e_1^kd_2e_3^kd_4e_5^k...
		\]
		has at least one pair of identical adjacent intervals of length $N$. 
	\end{Proposition}
	
	\begin{proof}
		We will show that for each $k$, we can defeat $e_{2i}^k$ with finitely many moves. The full statement follows by concatenating these moves together.\\
		We will compare the interval $x[1,N]$ with $x[N+1,2N]$. Thus, we need $a_{i} = a_{i+N}$ for $1\leq i \leq N$. For $1 \leq t \leq \frac{N}{2}$, let $d_{2t} = e_{2t +N}$ and let $d_{1 + N + 2t} = e_{2t + 1}$.  Then, for even $i = 2t$, we have that $a_i = d_{2t} = e_{2t + N} = a_{i+N}$ and for odd $i = 2t+1$ we have that $a_i = e_{2t+1} = d_{1 + N + 2t} = a_{2t + 1 + N}$.                                                                
	\end{proof}
	\subsection{Computing Proper Colorings of The Cayley Graphs of Lacunary Sequences}
	
	In this section, we effectivize an application of the Lefthanded local lemma due to \citeauthor{PeresSchlag2010} \cite{PeresSchlag2010}. 
	
	Let $\epsilon > 0$. A sequence $S = \{n_i\}_{i \in \bb{n}}$ is called a \textit{lacunary with parameter } $(1+\epsilon)$ if, for each $i \in \bb{n}$ we have that $n_{i+1}/ n_i > (1+\epsilon)$. We associate each lacunary sequence with its Cayley graph $\mc{G} = \mc{G}(S) = \mb{z}_S$, which by definition is the graph with vertex set $\bb{Z}$ such that $n$ and $m$ are connected by an edge if and only if $|n - m| \in S$. 
	
	For a graph $G$, the chromatic number $\chi(G)$ is the least number of colors one can assign to the vertices of $G$ such that all edges connect vertices of different colors, in what is known as a proper coloring. This is also a notion of non-repetitiveness sequences in the sense that being ``non-repetitiveness'' is determined by the edges of $G$. According to Katznelson \cite{Katznelson2001}, Paul Erd\H{o}s posed the question of whether $\chi(\mc{G})$ is finite for every lacunary sequence of any parameter greater than 1. Katznelson discovered the following connection to another question of Erd\H{o}s \cite{ErdosLovasz1975} asking whether, for each lacunary sequence $\{n_i\}_{i\in\mb{n}}$, there is a $\theta \in (0,1)$ such that $\{\theta n_i\}_{i\in\mb{n}}$ is not dense modulo 1. 
	
	\begin{Theorem}[\cite{Katznelson2001}]
		\label{colorCayley}
		Let $S = \{n_i\}_{i \in \mb{N}}$ be an integer valued sequence. Suppose there exists a $\delta > 0$ and $\theta \in (0,1)$ such that  $\inf_j\norm{\theta n_i} > \delta$, where $\norm{\cdot}$ denotes the distance to the nearest integer. Then, $\chi(\mc{G}(S)) \leq k := \lceil \delta^{-1} \rceil$. 
	\end{Theorem}
	\begin{proof}
		Partition the circle $[0,1)$ into $k$ many disjoint intervals $I_1,I_2,\dots,I_k$ of length $1/k \leq \delta$. Define coloring $c: \mb{n} \to k$ by $c(x) = j$ if and only if $\theta x \in I_j (\mod 1)$. Then, if $n - m \in S$, we have that $\norm{\theta(n-m)} >  \delta \geq 1/k$, so $c(n) \neq c(m)$. 
	\end{proof}
	
	Our knowledge of lower bounds on $\inf_j\norm{\theta n_j}$ (and hence, upper bounds on $\chi(\mc{G})$) have been improved over time and are of independent interest (for a more detailed history and applications, see \cite{Dubickas2006}). The first proofs that there is a $\theta$ such that $\inf_j\norm{\theta n_j}$ has any lower bound at all are due to de Mathan \cite{deMathan1980} and Pollington \cite{Pollington1979}, with a lower bound linear in $\epsilon^4 |\log \epsilon|^{-1}$. Katznelson \cite{Katznelson2001} improved this to being linear in $\epsilon^2 |\log \epsilon|^{-1}$, which was further improved to being linear in $\epsilon^2$ by \citeauthor{AkhunzhanovMoshchevitin2004} \cite{AkhunzhanovMoshchevitin2004}. Most recently, \citeauthor{PeresSchlag2010} improved this to being linear in $\epsilon |\log \epsilon|^{-1}$. 
		
	\begin{Theorem}[\cite{PeresSchlag2010}]
		\label{PeresSchlagThm}
		There exists a constant $c > 0$ such that the following holds. Let $S = \{n_i\}_{i \in \mb{N}}$ be a lacunary sequence with parameter $(1 + \epsilon)$. Then, there is a $\theta \in (0,1)$
		\[
			\inf_{j \geq 1} \norm{\theta n_j} > c \epsilon |\log \epsilon|^{-1}.
		\]
		
		Hence, the Cayley graph $\mc{G} := \mc{G}(S)$ satisfies $\chi(\mc{G}) < \lceil c^{-1} \epsilon^{-1} |\log \epsilon|\rceil $. 
	\end{Theorem}
	
	 Note that it is impossible to have a bound $\chi(\mc{G}) < c\epsilon^{-1}$, since for each $\epsilon < 1$, the there is a lacunary sequence with parameter $(1 + \epsilon)$ beginning with $1, 2, 3, \dots, \lfloor \epsilon^{-1} \rfloor$. Thus, Theorem \ref{PeresSchlagThm} is optimal up to the factor of $|\log \epsilon|^{-1}$. The proof of Theorem \ref{PeresSchlagThm} uses a simplified but less general version of the lefthanded LLL, proven independently from Pegden's result. 
	
	 We are interested in, for a given lacunary sequence $S = \{n_i\}_{i \in \mb{n}}$ with parameter $(1 + \epsilon)$, computing a proper coloring of $\mc{g}(S)$ with as few colors as possible. Note that the coloring in Theorem \ref{colorCayley} is uniformly computable with $\theta$ and $\delta$ as parameters. Thus, in order to compute a proper coloring of $\mc{G}(S)$, we must compute $\theta$ and a suitable $\delta$. 
	 
	 \begin{Theorem}[Effective version of Theorem \ref{PeresSchlagThm}]
	 	\label{effectivePeresSchlag}
	 	There is a constant $c >0$ and computable Turing functionals $\Theta$ and $\Phi$ such that the following holds. Let $S = \{n_i\}_{i \in \mb{N}}$ be a lacunary sequence with parameter $(1 + \frac{1}{n})$. Then, $\Theta(S,n) = \theta \in (0,1)$ such that
	 	\[
	 		\inf_{j \geq 1} \norm{\theta n_j} > c \epsilon |\log \epsilon|^{-1}.
	 	\]
	 	
	 	Furthermore, $\Phi(\Theta(S,n),\epsilon)$ is a coloring $c : \mb{Z} \to \lceil c^{-1} \epsilon^{-1} |\log \epsilon| \rceil$ is a proper coloring of $\mc{G}(S)$.  
	 \end{Theorem}
	 \begin{proof}
	 	We effectivize Peres and Schlag's proof using Theorem \ref{CLLLL}. We show that the theorem is true with $c = 1/360$. Let $M > 5$ be greater than the doubling time of $S$, that is, such that $n_{j +M} > 2n_j$ for all $j$. For example, we can chose $M > \max(\log_{1 + \epsilon} (2),5)$. Let 
	 	\[
	 	\delta =  \frac{c}{M \log_2 M}.
	 	\]
	 	For each $n_j$, let 
	 	\[
	 	E_j = \left\{ \theta \in [0,1): \norm{\theta n_j} < \delta \right\}.
	 	\]
	 	
	 	Let $\mc{X} =\{x_1,x_2,\dots\}$ be a countable collection of mutually independent identically distributed random variables such that $\Pr(x = 0) = \Pr(x = 1) = \frac{1}{2}$. As usual, let $x[m,n)$ denote the set $\{x_m,x_{m+1},\dots, x_{n_1}\}$. We construct a valuation $a_1,a_2,\dots$ of $\mc{X}$ such that $0.a_1a_2a_3,\dots$ is a binary representation of
	 	\[
	 		\theta \in \bigcap_{j \geq 1} E_j.
	 	\]
	 	
	 	Note that each $E_j$ is the union of $n_j$ many uniformly spaced intervals (in the circle topology) $I^j_1, I^j_2,\dots I^j_{n_j}$  of length $2\delta/n_j$. Let $\ell_j$ be the unique element of $\mb{n}$ such that 
	 	\[
	 	2^{-\ell_j} \geq \frac{2\delta}{n_j} > 2^{-\ell_j + 1}.
	 	\]
	 	Then, each $I^j_i$ is covered by at most two dyadic intervals of measure $2^{-\ell_j}$. For each $n_j$, let $A_j$ be the set of all dyadic intervals of measure $2^{-\ell_j}$ which have nonempty intersection with $E_j$. Then, $\vbl(A_j) = x[1,\ell_j]$ and satisfaction of $A_j$ is sufficiently computable for Theorem \ref{CLLLL}. Let $C_1 = 6$ and let 
	 	\[
	 	p = \lceil C_1 \log_2 M \rceil,
	 	\]
	 	be the number of variables we will resample from $\vbl(A_j)$. Let $\rsp(A_j) = x(\min\{0,\ell_j - p\},\ell_{j}]$. Then, each $x_i$ is in finitely many $\rsp$ sets which can be listed uniformly in $i$, so the computability conditions of Theorem \ref{CLLLL} are satisfied. We set $A_i \prec A_j$ whenever $i < j$. Since $a_{n + M}/a_n \geq 2$ for all $n$, we have that, for each $n_j$ there are at most $M$ many $n_i$ such that $\ell_i = \ell_j$. Hence, $|\nbr(A_j)| \leq 2Mp$. 
	 	
	 	We now compute a suitable $P^*(A_j)$. Consider a dyadic interval $L$ of measure $2^{-\ell_j + p}$. Then our valuation $\theta = 0.a_1a_2\dots$ being in $L$ corresponds to the initial $a_1,a_2,\dots, a_{\ell_j - p}$ of length $\ell-p$ matching $L$, so it is sufficient to find $P^*(A_j)$ such that 
	 	\[
	 	P^*(A_j)  \geq \Pr(A_j | L)
	 	\] 
	 	for each dyadic interval $L$ of measure $2^{-\ell_j + p}$.
	 	
	 	There are at most $\lfloor n_j / 2^{\ell_j - p}\rfloor + 1$ many $I^j_i$ that have nonempty intersection with $L$. Each of these has nonempty intersection with at most two dyadic intervals of measure $2^{-\ell_j}$. Thus,

	 	\begin{align}
	 	\Pr(A_j | L) &\leq \frac{2\left(1 + \frac{n_j}{2^{\ell_j - p}}\right)2^{-\ell_j}}{2^{-\ell_j + p}} \nonumber\\
	 		&= 2\left(2^{-p} + \frac{n_j}{2^{\ell_j}}\right) \nonumber\\
	 		&\leq 2\left(\frac{1}{2^{\lceil C_1 \log_2 M\rceil}} + 4 \delta\right) \nonumber \\
	 		& \leq 2M^{-C_1} + 8\delta. \label{lineWithM}
	 	\end{align}
	 	
	 	By our choice of $C_1 = 6$,  $c = 1/360$ and $M \geq 5$, we have that 
	 	
	 	\begin{equation}
	 		\frac{c_0 M^{C_1}}{M \log_2 M} \geq 1 \label{constantRequirement}
	 	\end{equation}

	 	and hence that $M^{-C_1} \leq \delta$. Applying this bound to Inequality $\ref{lineWithM}$ yields 
	 	\[
	 		\Pr(A_j | L) \leq 10\delta,
	 	\]
	 	so we may set $P^*(A_j) = 10\delta$. 
	 	
	 	We move to establishing the local lemma condition. Assume $i < j$. Then, we have that $A_i \in \nbr(A_j)$ if and only if there is $k$ such that $\ell_k = \ell_i$ and $j - p \leq k$. Since there are at most $M$ many $k$ such that $\ell_i = \ell_k$ we have that $|\nbr(A_j)| \leq 2pM$. 
	 	
	 	Let $h = 2pM$ and set $z(A_j) = h^{-1}$. Then,
	 	\[
	 		z(A_j) \prod_{B \in \nbr(A_j)} (1 - z(A_j)) \leq \frac{1}{h} \left(1 - \frac{1}{h}\right)^{h}.
	 	\]
	 	
	 	We need to show that 
	 	\begin{equation}
	 		h^{-1}(1-h^{-1})^h \geq P^*(A_j) = 10\delta. \label{hfactor}
	 	\end{equation}
	 	
	 	Since $h > 5$, we have that $(1 - h^{-1})^h \geq 1/3$. Hence, Condition \ref{hfactor} becomes
	 	
	 	\begin{equation}
	 		h^{-1} \geq 30\delta. \label{hAndDelta}
	 	\end{equation}
	 	Since $C_1 \log_2 M > 12$, we have
	 	\[
	 	  h\delta = \frac{2c \lceil C_1 \log_2 M \rceil}{\log_2 M} \leq \frac{13C_1c}{12},
	 	\]
	 	
	 	so Condition $\ref{hAndDelta}$ becomes 
	 	\[
	 	1 \leq \frac{720 C_1 c}{13}.
	 	\]
	 	Which, along with Condition \ref{constantRequirement}, is satisfied by $C = 6, M \geq 5$, and $c = 1/360$.

	 \end{proof}
	 
	 Note that the $c$ in Theorem \ref{PeresSchlagThm} can be chosen to be $240$, as opposed to our choice of $c = 360$ in Theorem \ref{effectivePeresSchlag}. This is due to the symmetry of the neighborhood relation $\nbr$ of Theorem \ref{CLLLL}, since it is defined based on sharing resample variables. On the other hand, the neighborhood relation of Theorem \ref{LLLL} is much more flexible and not symmetric. In the present case, the neighborhoods of Theorem \ref{effectivePeresSchlag} are twice as large as those in Theorem \ref{PeresSchlagThm}. Despite this discrepancy, the result on the asymptotic behavior of $\chi(\mc{g}(S))$ is maintained. In contrast, the gap between Theorems \ref{ComputableThueChoice} and \ref{ThueChoice} is more interesting. 
	 
	\subsection{Effective Thue Choice Number of the Path} 
	In this section, we apply the CLLLL (Theorem \ref{CLLLL}) to an application of the LLLL (Theorem \ref{LLLL}) due to \citeauthor{GrytczukPryzybyloZhu2011} \cite{GrytczukPryzybyloZhu2011}. 
	
	We say that a string is a square if it is of the form $xx$ for some word $x$. For example, \textit{hotshots} is a square. A string or sequence is square free if none of its substrings are squares. For example, \textit{cab\underline{acac}ba} and \textit{abcac\underline{bb}a} are not sqaure free. As previously mentioned, Thue's construction of a square free ternary sequence is said to be the starting point of a longstanding investigation into similar phenomena. \textit{List colorings} are such an example. Consider the path graph $P$ isomorphic to $\N$. We say that $L: P \to \mc{p}(\N)$ is a \textit{list assignment} of $\N$ if each $L(v)$ is non-empty. We think of $L$ as assigning each $v \in P$ a set $L(v) \subset \N$ of colors which we may use to color $v$. An $L$-coloring of $P$ is a coloring $c: P \to \N$ such that for each $v \in P$ we have that $c(v) \in L(v)$. The \textit{Thue choice number} of the path is the smallest number $n$ such that, whenever $|L(v)| = n$ for all $v \in P$, there is a square free $L$-coloring of $P$. We denote the Thue choice number of the path by $\pi_{\ch}(P)$. Since there are no square free binary sequences we have that $\pi_{\ch}(P) > 2$. 
	
	\begin{Theorem}[\cite{GrytczukPryzybyloZhu2011,GrytczukKozikMicek2013,Rosenfeld2020}]
		\label{ThueChoice}
		$\pi_{\ch}(P) \leq 4$.
	\end{Theorem}
	
	\citeauthor{GrytczukPryzybyloZhu2011} use the lefthanded local lemma to give the chronologically first proof that $\pi_{\ch}(P) \leq 4$. Their proof inspired relatively simpler proofs by \citeauthor{GrytczukKozikMicek2013} \cite{GrytczukKozikMicek2013} and Rosenfeld \cite{Rosenfeld2020}. Furthermore, Rosenfeld has more recently showed that, if $|\bigcup_{v \in P} L(v)| \leq 4|$ and $|L(v)| = 3$ for all $v \in P$ then there is a square free $L$-coloring of $P$ \cite{Rosenfeld2022}.
	
	It is natural to ask whether one can compute such a coloring from $L$. Let $\pi_{\cch}(P)$ be the least $n$ such that for each $|L(v)| = n$ for each $v \in P$ there is a $L$-computable $L$-list coloring of $P$. It is immediate that $\pi_{\ch}(P) \leq \pi_{\cch}(P)$. 
	
	Unlike for the other applications of the LLLL, the proof that $\pi_{\ch}(P) \leq 4$ requires optimization. However, due to the neighborhoods of events in our application of Theorem~\ref{CLLLL} being larger than those in \citeauthor{GrytczukPryzybyloZhu2011}'s application of Theorem~\ref{LLLL}, we can only show that $\pi_{\cch}(P) \leq 6$.  

	\begin{Theorem}
		\label{ComputableThueChoice}
		$\pi_{\cch}(P) \leq 6.$
	\end{Theorem}
	\begin{proof}
		Let $L: \N \to \mc{p}(\N)$ be given such that each $|L(i)| = 6$. 
		
		We describe a random process computable from $L$ to construct a sequence $a_1,a_2,\dots$ such that there are no squares of sizes 1 or 2. We will apply Theorem \ref{CLLLL} to compute a sequence which is square free.
		
		Stage 1: Let $a_1$ be chosen uniformly at random from $L(1)$. 
		
		Stage $i$ for $i \in \{2,3\}$: Let $a_{n+1}$ be chosen uniformly at random from $L(i) \setminus \{a_{i-1}\}$.
		
		Stage $n+3$: Suppose that $a_1,\dots a_{n+2}$ have already been chosen. If $a_n \neq a_{n+2}$ then let $a_{n+3}$ be chosen uniformly at random from $L(n+3) \setminus \{a_{n+2}\}$. If $a_n = a_n+2$ then let $a_{n+3}$ be chosen uniformly at random from $L(n+3) \setminus \{a_{n+1},a_{n+2}\}$.  
		
		As usual, for $i < j$, we let $a[i,j)$ denote the string 
		
		\[a[i,j) = a_{i}, a_{i+1}, \dots, a_{j-1}.\]

		To apply Theorem \ref{CLLLL}, define $A_{k,n}$ to be the event that 
		\[
			a[k,k+n) = a[k+n,k+2n)
		\]
		and say that $n$ is the \textit{size} of $A_{k,n}$. 

		We can define variables with finite ranges $\mc{X} = \{x_1,x_2,\dots\}$ such that the random variable $a_i$ is determined by $\{x_j: j \leq i\}$. For example, we could have our $x_i$ take values which are ordered triples over the Cartesian product $6 \times 5 \times 4$, with the value of each coordinate corresponding to which element of $L(i)$ is chosen in the different cases of stage $n+3$. Note that in this setup, $A_{k,n}$ is not in the sigma algebra generated by $x[k,k+2n)$. However, we can set 
		\[
		\rsp(A_{k,n}) = x[k+n,k+2n)
		\]
		and $P^*(A_{k,n}) = \frac{1}{4^n}$ to satisfy the $P^*$ Condition~\ref{PStarCondition}. Furthermore, all of the computability conditions of Theorem \ref{CLLLL} are satisfied and the linear order is supplied by Lemma~\ref{IntervalsWork}.

		To verify the LLL Condition, we set $z(A_{k,n}) = \frac{1}{n^3}$. Note that $A_{k_0,n_0}$ has at most $n + n_0$ many neighbors of size $n$. Since the random process never produces squares of sizes 1 or 2, we can can ignore $A_{k,n}$ with $n \in \{1,2\}$.  Hence,
		
		\[
		\begin{split}
			z(A_{k_0,n_0}) \prod_{B \in \nbr\left(B_{k_0,n_0}\right)} \left(1 - z(B)\right) &\geq \frac{1}{n_0^3} \prod_{n \geq 3}\left(1 - \frac{1}{n^3}\right)^{n + n_0} \\
			&\geq \frac{1}{n_0^3} \left(\prod_{n \geq 3}\left(1 - \frac{1}{n^3}\right)^{n_0}\right)\left(\prod_{n \geq 3} \left(1 - \frac{1}{n^2}\right)\right)\\
			&\geq \frac{1}{n_0^3} \left(1 - \sum_{n \geq 3}\frac{1}{n^3}\right)^{n_0} \left( 1 - \sum_{n \geq 3} \frac{1}{n^2}\right)\\
			&\geq \frac{1}{n_0^3}(0.9229)^{n_0}(0.605)\\
			&\geq \frac{1}{4^{n_0}} = P^*(E_{k_0,n_0}).
		\end{split}
		\]
		
		To check the last inequality for $n \geq 3$, let $f(x) = \frac{0.605 \cdot (0.9229)^x}{x^3}$ and $g(x) = \frac{1}{4^x}$. Then, 
		\[\log(f(x) - g(x)) = -3 \log(x) + x (\log(0.9229) - \log(0.25)) + \log(0.605).
		\] 
		The equation $\log(x) =cx$ has exactly one solution $s$ for $c > 0$. Since $cx > \log(x)$ for $x > s$, it suffices to show that $f(3) - g(3) > 0$ and that $f'(3) - g'(3) > 0$. Indeed, $f(3) - g(3) \approx 0.0019 > 0$ and $f'(3) - g'(3) \approx 0.0026 > 0$, so the local lemma Condition~\ref{LeftLocalLemmaCondition} is satisfied. 	
	\end{proof}	
	
	\section{Conclusion}
	
	We conclude with a summary of questions arising from this work.
	
	While we can computably defeat any Player 2 strategy in the binary sequence games, it is still unknown whether the indeterminacy argument that pastes these winning moves into a full strategy is computable. 
	
	\begin{Question}
		Are there computable winning strategies for the binary sequence games in Theorems \ref{BeckComputableGame} and \ref{AlonComputableGame}?
	\end{Question}
	
	In each of our applications, there is a numerical gap between what can be shown non-constructively with Theorem~\ref{LLLL} and what can be shown constructively with Theorem~\ref{CLLLL}. This is due to the fact that the graph induced by the neighborhood relation in Theorem~\ref{CLLLL} is undirected, while Theorem~\ref{LLLL} allows directed graphs. As a result, the neighborhoods in Theorem~\ref{CLLLL} are typically twice as large as those in Theorem~\ref{LLLL}. The most salient consequence of this is that Theorem~\ref{ComputableThueChoice} is unable to get as close to the lower bound of $\pi_{\ch}(P) \geq 3$ as Theorem~\ref{ThueChoice} does. 
	
	\begin{Question}
		\label{DiffThueNumbers?}
		Is $\pi_{\ch}(P)$ strictly less than $\pi_{\cch}(P)$? 
	\end{Question}
	
	One way to bring the bound on $\pi_{\cch}(P)$ down would be to prove a stronger version of Theorem~\ref{CLLLL} that more closely resembles Theorem~\ref{LLLL} in allowing the graph induced by the neighborhood relation to be undirected. However, this would require a different analysis - there are legal logs in which $E_j \succ E_{j+1}$ and $\rsp(E_j)\cap \rsp(E_{j+1}) \neq \emptyset$. One possibility is to analyze logs going forward in time, as in \cite{Harvey2020}, as opposed to the backward-in-time analysis of Moser trees. Even if this is possible in the algorithmic case for finite witnesses, a new argument would be required to obtain a computable version. 
	
	\begin{Question}
		\label{FullLLLL?}
		Is Theorem~\ref{LLLL} computably true in full generality?
	\end{Question}
	
	One way to answer Questions~\ref{DiffThueNumbers?} and \ref{FullLLLL?} would be to demonstrate that $\pi_{\cch}(P) \leq 4$. 
	
	Saying that a sequence is square free is equivalent to saying that it avoids the pattern $xx$.  We can thus generalize Question~\ref{DiffThueNumbers?}.
	
	\begin{Question}
		Does pattern avoidance contain computational content?
	\end{Question}
	Here, we are purposely vague, as pattern avoidance can be generalized to, for example, list colorings of graphs. In contrast to the situation of pattern avoidance, the theory of forbidden \textit{words} is well studied. \citeauthor{CenzerDashtiKing2008} \cite{CenzerDashtiKing2008} showed that there is a c.e. sequence of binary strings $\sigma_1,\sigma_2,\dots$ such that the set of sequences $c$ not containing any $\sigma_i$ as a substring form a non-empty $\Pi^0_1$ class with no computable element. See also \cite{Miller2012,CallardSalomonVanier2024}. While seemingly related, it is not clear what link, if any, there is between the theory of non-repetitiveness (pattern avoidance) and the theory of subshifts (forbidden words). 
	
\newpage
\printbibliography

\end{document}